%% file: Stutte_I_120.tex
\documentclass[a4paper,12pt]{article}
\usepackage[utf8]{inputenc}
\usepackage{a4wide}
\usepackage{amsmath, amsthm}
\usepackage{amssymb}
\usepackage{graphicx}
\usepackage{ifpdf}
\usepackage{hyperref}
\usepackage{comment}
\usepackage{mathtools}
\usepackage{arydshln}

\theoremstyle{plain}
\newtheorem{theorem}{Theorem}[section]
\newtheorem{lemma}[theorem]{Lemma}
\newtheorem{proposition}[theorem]{Proposition}
\newtheorem{problem}[theorem]{Problem}

\newtheorem{corollary}[theorem]{Corollary}

\newtheorem{definition}[theorem]{Definition}

\theoremstyle{definition}
\newtheorem{example}[theorem]{Example}

\theoremstyle{remark}
\newtheorem{remark}[theorem]{Remark}

\newcommand{\reg}{\text{reg}}
\newcommand{\eps}{\varepsilon}
\newcommand{\tr}{\text{tr}}

\title{A Tutte polynomial for maps}
\author{Andrew Goodall\thanks{Charles University, Prague, Czech Republic. Email: \texttt{andrew@iuuk.mff.cuni.cz}. Supported by Project ERCCZ LL1201 Cores and the Czech Science Foundation, GA \v{C}R 16-19910S.} \and
Thomas Krajewski\thanks{Aix-Marseille Universit\'e, Universit\'e de Toulon, CNRS, CPT, Marseille, France}\thanks{Supported by the ANR grant ANR JCJC ``CombPhysMat2Tens"}
\and Guus Regts\thanks{University of Amsterdam,  Netherlands. Email: \texttt{guusregts@gmail.com}. Supported by the European Research Council under the European Union's Seventh Framework Programme (FP7/2007-2013) / ERC grant agreement n$\mbox{}^{\circ}$ 339109, and by a NWO Veni grant.}\and Llu\'is Vena\thanks{Charles University, Prague, Czech Republic.  Email: \texttt{lluis.vena@gmail.com}.  Supported by the Center of Excellence-Inst for Theor. Comp. Sci., Prague, P202/12/G061, and by Project ERCCZ LL1201 Cores.}}

\begin{document}
\maketitle
\begin{abstract}
We follow the example of Tutte in his construction of the dichromate of a graph (that is, the Tutte polynomial) as a unification of the chromatic polynomial and the flow polynomial in order to construct a new polynomial invariant of maps (graphs embedded in orientable surfaces). We call this the surface Tutte polynomial. The surface Tutte polynomial of a map contains the Las Vergnas polynomial, Bollob\'as-Riordan polynomial and Kruskhal polynomial as specializations.
By construction, the surface Tutte polynomial includes among its evaluations the number of local tensions and local flows taking values in any given finite group. Other evaluations include the number of quasi-forests. 
%Moreover, the surface Tutte polynomial of a plane map specializes to the Tutte polynomial of the underlying planar graph while for non-plane maps there is no such specialization.
\end{abstract}

\tableofcontents

\input{intro_st_120.tex}

\input{tutte_graphs_maps_s2_s3_ini_120.tex}

\section{Specializations}\label{sec:spec}

\input{specializations_st_120.tex}

% %\input{indep_tutte_101.tex}
% 
% 

%\input{to_bouquets_100.tex}

%%%%%%%%%%%%%%%%%%%%%%%%%%%%%%%%%%%%%
%%%%%%    Proofs   %%%%%
%%%%%%%%%%%%%%%%%%%%%%%%%

\input{proofs_100.tex}

%%%%%%%%%%%%%%%%%%%%%%%%%%%%%%%%%%%%%%%%%
%%%%%%%%%%%%%                  %%%%%%%%%%
%%%%%%%%%%%%%    Conclusions  %%%%%%%%%%%
%%%%%%%%%%%%%                  %%%%%%%%%%
%%%%%%%%%%%%%%%%%%%%%%%%%%%%%%%%%%%%%%%%%

\input{conclusions_120.tex}

\bibliographystyle{plain}
\bibliography{SurfaceTutte}
\end{document}

%% file: intro_st_120.tex
\section{Introduction}

Inspired by Tutte's construction of the dichromate of a graph, we construct a similarly defined polynomial invariant of maps (graphs embedded in an orientable surface), which we call the {\em surface Tutte polynomial}. For a plane map the surface Tutte polynomial is essentially the Tutte polynomial of the underlying planar graph; for non-plane maps it includes the Tutte polynomial as a specialization. Moreover the surface Tutte polynomial includes as a specialization the Las Vergnas polynomial, Bollob\'as-Riordan polynomial and Kruskhal polynomial of an graph embedded in an orientable surface.
%in a sense made precise in Section~\ref{sec:independence} below.
The surface Tutte polynomial has evaluations that count local flows and tensions of a map taking values in a finite nonabelian group, comparable in this way to the Tutte polynomial, which has specializations counting abelian flows and tensions of a graph. We also give some other topologically significant evaluations of the surface Tutte polynomial, such as the number of quasi-forests.

\subsection{Colourings, tensions and flows of graphs}\label{sec:tensions_flows_graphs}
Tutte~\cite{T54, T04} defined the {\em dichromate} of a graph $\Gamma$ (later to become known as the Tutte polynomial) as a bivariate generalization of the chromatic polynomial of $\Gamma$ and the flow polynomial of $\Gamma$.   The Tutte polynomial can be defined more generally for matrices (not just adjacency matrices of graphs), and in greater generality for matroids.  
The reader is referred to~\cite{B93, W93, B98, W99, EMM11} for more on the Tutte polynomial.

%Birkhoff~\cite{B12, B34} introduced, flow poly Tutte~\cite{T54} 
The chromatic polynomial of a graph $\Gamma$ evaluated at a positive integer $n$ counts the number of proper colourings of $\Gamma$ using at most $n$ colours. 
The flow polynomial of $\Gamma$ evaluated at a positive integer $n$ counts the number of nowhere-zero $\mathbb Z_n$-flows of $\Gamma$ (assignments of non-zero elements of $\mathbb Z_n$ to the edges of $\Gamma$ with a fixed arbitrary orientation so that Kirchhoff's law is satisfied at each vertex). Proper colourings of $\Gamma$ can be described in terms of nowhere-zero $\mathbb Z_n$-tensions of $\Gamma$, which are assignments of non-zero elements of $\mathbb Z_n$ to the edges of $\Gamma$ with the property that for every closed walk in the oriented graph $\Gamma$ the sum of values assigned to the forward edges equals the sum of values on the backward~edges.
%Given a proper colouring of $\Gamma$ using colour set $\mathbb Z_k$, assigning the difference between colours on the endvertices of an edge (using the edge orientation to determine in which order to take the difference of endpoint colours) forms a nowhere-zero $\mathbb Z_k$-tension of $\Gamma$.

%The dichromate of $G$ subsequently became known as the Tutte polynomial, denoted by  $T(G;x,y)$ on the hyperbola $(x-1)(y-1)=k$ meets along $y=0$  and along $x=0$ its specializations to the chromatic polynomial  and the flow polynomial of $G$ when evaluated at $k$, counting respectively proper $k$-colorings\index{$k$-coloring!proper} and nowhere-zero $\mathbb{Z}_k$-flows of $G$ when $k$ is a positive integer. 

Tutte showed~\cite{T54} that flows and tensions of a graph $\Gamma$ using non-zero values from a finite additive abelian group $G$ of order $n$ are in bijective correspondence with nowhere-zero $\mathbb Z_n$-flows and nowhere-zero $\mathbb Z_n$-tensions of $\Gamma$,  counted by $|T(\Gamma;0,1-n)|$ and $|T(\Gamma;1-n,0)|$, respectively.  
While colourings and flows are dual notions for planar graphs -- the $G$-tensions of a plane graph correspond to the $G$-flows of the dual plane graph -- for general graphs duality resides at the level of the cycle and cocycle matroids of the graph: the module of $G$-tensions of a non-planar graph does not correspond to the module of $G$-flows of a graph. (This follows from Whitney's matroid characterization of planar graphs~\cite{W32a}.) %, this primal-dual correspondence does not extend beyond the planar case. %, but rather of a cographic matroid. %There is a consequent asymmetry between properties of colorings and flows for graphs (for example, loopless graphs may have arbitrarily high chromatic number, but bridgeless graphs have bounded flow number).

%Our motivation for this paper was to define a polynomial for embedded graphs that, like Tutte's dichromate, evaluated at certain points gives the number of tensions and flows.
\subsection{Nonabelian flows and tensions of maps}
  %We extend the problem of counting tensions and flows on a graph taking values in abelian groups to nonabelian groups. 
In the combinatorial literature nonabelian flows seem only to have been considered by DeVos~\cite{D00}. %, while in the representation theory literature there are immediately transferable results on which we shall draw in this paper~\cite{M78, J98}. %The recent interaction of graph theory and representation theory (see for example~\cite{S15}) in the framework of Lie algebras also plays a part in a direct derivation of the number of nonabelian flows and tensions, which we have not included here in order to keep the length of the paper to a minimum. Our perspective is to use an algebraic structure associated with a nonabelian group $G$ to describe an invariant of a graph $\Gamma$. An inverse approach would be to use a graph $\Gamma$ in order to describe an invariant for algebraic structures.
There are two significant differences between the abelian and the nonabelian case for defining flows and tensions for a graph $\Gamma$, which are most easily illustrated by considering~flows. %Recall that the defining condition for a flow of a graph is that, relative to an arbitrary orientation of its edges, Kirchhoff's law is satisfied at each vertex:  the sum of values on outgoing edges must equal the sum of values on ingoing edges.

The first is that for an abelian group the Kirchhoff condition for a flow requires that the sum of values on incoming edges is equal to the sum of values on outgoing edges where elements can be added together in any order. However, for a nonabelian group the order in which the group elements are composed matters: the edges incident with a common vertex need to be ordered so as to ensure the flow condition is well-defined. This order on edges around a vertex only matters up to cyclic permutation. Therefore the first change going from abelian to nonabelian flows is to attach a cyclic order of edges incident with a common vertex. This is equivalent to specifying an embedding of the graph in an orientable surface. Here we are moving from graphs to~maps.

The second difference is that, for an abelian group, if the Kirchhoff condition for a flow is satisfied at each vertex then, for any %subset $U$ of vertices, the sum of values on edges going from $U$ to outside $U$ is equal to the sum of values on edges going from outside $U$ into $U$. In other words, 
cutset of edges (not necessarily defined by a single vertex), the flow values in one direction have the same sum as the flow values in the reverse direction. 
The same is not necessarily true for an assignment of values to edges from a nonabelian group: Kirchhoff's condition may be satisfied at each vertex, but there may be edge cutsets for which the product of values on outgoing edges does not equal that on incoming edges (no matter in which order the edges are taken).  
Indeed, it is not clear for a general graph $\Gamma$ how to define a cyclic order of edges in an arbitrary cutset using just the cyclic orders of edges around vertices given by an orientable embedding of $\Gamma$. 
%In this paper we do not take the step of appending a cyclic order to every cutset in a graph, but just assign a cyclic order to edges incident with a common vertex (vertex rotations). 
Flows on maps are thus defined locally, in the sense 
that the Kirchhoff condition is satisfied at each vertex, but does not necessarily extend to other edge cutsets.  %In order to study globally defined flows the underlying graph $\Gamma$ would need to be enriched by 
%As additional structure of cyclic orders on all cutsets would otherwise be required.

The dual case to flows is that of tensions taking values in a nonabelian group. For a local tension, rather than requiring for every closed walk  that product of values taken in order around the walk, in which values on backward edges are inverted, is equal to the identity, we just require this condition for facial walks.  % (equivalently, on contractible cycles). 
For local tensions, the cyclic order of edges around a face is given by taking the edges in the order given by walking around it. In distinction to the case of cutsets, given a cycle of a graph there is a cyclic order of edges already determined by walking around it -- there need not be any embedding of the graph as a map. Thus it is possible to define a global tension of a graph taking values in a nonabelian group, in the sense that the product of values around any cycle must make the identity (values on backward edges are inverted). Just as for abelian groups, such global tensions correspond to vertex colourings of the graph. Local tensions do not have this correspondence, except in the case of planar graphs. However, the dual of a global nonabelian tension is not defined for graphs except in the case of planar graphs, whereas the dual of a local nonabelian tension of a map is a local nonabelian flow of a map.

\subsection{The surface Tutte polynomial}\label{sec:intro_surface_Tutte}
We now turn to the definition of the surface Tutte polynomial of a map.
For a map $M$ that is a 2-cell embedding of a graph $\Gamma$ in an orientable surface, we let $v(M)$, $e(M)$, $k(M)$ equal respectively the number of vertices, edges and connected components of $\Gamma$  (each connected component is embedded its own surface), $f(M)$ the sum of the number of faces in the embeddings of each component of $\Gamma$, and $g(M)$ the sum of the genera of the surfaces in which the connected components of $\Gamma$ are embedded.   
We define the {\em surface Tutte polynomial} of a map $M$ (see Definition~\ref{def:surface_tutte} below) to be the multivariate polynomial in variables $x,y,\mathbf x=(x_0,\dots, x_{g(M)}), \mathbf y=(y_0,\dots, y_{g(M)})$ given by 
$$\mathcal T(M;\mathbf x,\mathbf y)=\sum_{A\subseteq E}x^{e(M\!/\!A)-\!f(M\!/\!A)+\!k(M\!/\!A)}y^{e(M\!\backslash\! A^c)-\!v(M\!\backslash\! A^c)+\!k(M\!\backslash\! A^c)}\prod_{\stackrel{\mbox{\rm \tiny conn. cpts $M_i$}}{\mbox{\rm \tiny of $M/A$}}}\!\!x_{g(M_i)}\prod_{\stackrel{\mbox{\rm \tiny conn. cpts $M_j$}}{\mbox{\rm \tiny of $M\backslash A^c$}}}\!\!y_{g(M_j)},$$ 
where $A^c=E\setminus A$ for $A\subseteq E$ , $M/A$ is the map obtained by contracting the edges in $A$ and $M\backslash A^c$ is the map restricted to edges in $A$.  %, $M^*$ denotes the geometric dual of the map $M$, 
%and $n(M)=e(M)-v(M)+k(M)$ denotes the nullity of $M$.

We show in Section~\ref{sec:proofs} that the surface Tutte polynomial has evaluations counting nowhere-identity nonabelian local flows and tensions.
It moreover contains as a specialization other polynomial invariants that have been defined for maps, notably the Las Vergnas polynomial~\cite{LV78, LV80, EMM15}, Bollob\'as--Riordan polynomial~\cite{BR01, BR02} and the Krushkal polynomial~\cite{K11} (see Section~\ref{sec:relation} below). 
There are maps $M$ with different surface Tutte polynomials but equal Krushkal polynomials (see Section~\ref{sec:not_Kruskal} below). 
The fact that the surface Tutte polynomial has evaluations counting nowhere-identity nonabelian local flows and tensions recommends it as a natural translation of the Tutte polynomial for graphs (in its guise as the dichromate) to a Tutte polynomial for maps.

%A further generalization is to define a multivariate polynomial recursively for deletion and contraction of ordinary edges, and to specify boundary value $z_{i,j}$ on a connected graph with $i$ bridges and $j$ loops and no ordinary edges. This leads to the strong $U$-polynomial of Merino and Noble~\cite[Section 4.2]{MN09}, which is equivalent to what is called in that paper the strong Tutte symmetric function, and has a subgraph expansion of the form
%$$\overline{U}(\Gamma;\mathbf z)=\sum_{A\subseteq E}\quad\prod_{\stackrel{\mbox{\rm \tiny conn. cpts $C_i$}}{\mbox{\rm \tiny of $\Gamma\backslash A$}}}z_{r(C_i), n(C_i)},$$
%where $r(C_i)=|V(C_i)|-1$ is the rank of the connected component $C_i$ of $\Gamma\backslash A$. 
%
%For this polynomial, however, a deletion-contraction recurrence requires extending the polynomial to vertex-weighted graphs (called a strong $W$-function), with the replacement of $|V|$ by sum of vertex weights in defining to rank and nullity of a vertex-weighted graph $(V,E)$.  
%The $U$-polynomial is obtained from $\overline{U}$ by the specialization $z_{i,j}=x_i(y-1)^j$, and the Tutte polynomial from the $U$-polynomial by the specialization $x_i=x-1$.  

Setting $x_g=a^{g}$ and $y_g=b^{g}$ in $\mathcal T(M;\mathbf x,\mathbf y)$ for $g=0,1,\dots, g(M)$ gives the following quadrivariate polynomial specialization of the surface Tutte polynomial (Definition~\ref{def:S}):
$$\mathcal Q(M;x,y,a,b)=\sum_{A\subseteq E}x^{v(M\!/\!A)-k(M\!/\!A)+2g(M\!/\!A)}y^{e(M\!\backslash\! A^c)-\!v(M\!\backslash\! A^c)+\!k(M\!\backslash\! A^c)}a^{g(M\!/\!A)}b^{g(M\!\backslash\! A^c)}.$$
The polynomial $Q(M;x,y,a,b)$ may be defined not just for maps but for $\Delta$-matroids more generally. %This $\Delta$-matroid invariant differs from those considered in~\cite{CMNR16}, which are obtained by enlarging the domain of the Krushkal polynomial from maps to $\Delta$-matroids.\footnote{A: Modify this sentence once we know whether we can or indeed distinguish $\mathcal{Q}(M)$ from $\mathcal K(M)$.}\footnote{L: I have a strong feeling that they force the same partition on maps.}

\begin{remark}\label{rmk:Tutte_V-function} The surface Tutte polynomial of a map is in an unbounded number of variables, just as is the case for
Tutte's  $V$-function~\cite{T47, T54} of a graph $\Gamma=(V,E)$. %We require these variables in order to be able to evaluate the number of nonabelian flows and tensions of a map. %In a similar way to how Tutte's $V$-function includes the Tutte polynomial as a specialization, which is extendable to matroids, so the surface Tutte polynomial contains the quadrivariate polynomial $\mathcal Q(M)$ as a specialization, which is extendable to $\Delta$-matroids.   
Tutte's universal $V$-function is a polynomial in a sequence of commuting indeterminates $\mathbf y=(y_0,y_1,\dots)$ 
%defined recursively as follows:
%\begin{itemize}
%\item[(1)] If %$\Gamma\cong \Gamma_1\sqcup \cdots \sqcup \Gamma_{k(\Gamma)}$, where each $\Gamma_i$ is 
%$\Gamma$ is a connected graph that has $j$ loops and no other types of edges (so it is a single vertex with $j$ loops attached) then
%%with $\ell_i$ loops on the vertex  in the $i$th component $\Gamma_i$ of $\Gamma$ , then 
%$$V(\Gamma;\mathbf y)=y_{j}.$$
%\item[(2)] For any edge $e$ of $\Gamma$ that is not a loop, 
%$$V(\Gamma;\mathbf y)= V(\Gamma/e;\mathbf y)+V(\Gamma\backslash e;\mathbf y).$$
%\item[(3)] If $\Gamma=\Gamma_1\sqcup\Gamma_2$ then 
%$$V(\Gamma_1\sqcup\Gamma_2;\mathbf y)=V(\Gamma_1;\mathbf y)V(\Gamma_2;\mathbf y).$$
%\end{itemize}
defined by the subgraph expansion
$$V(\Gamma;\mathbf y)=\sum_{A\subseteq E}\quad\prod_{\stackrel{\mbox{\rm \tiny conn. cpts $C_i$}}{\mbox{\rm \tiny of $\Gamma\backslash A$}}}y_{n(C_i)},$$
where $n(C_i)=|E(C_i)|-|V(C_i)|+1$ is the nullity of the $i$th connected component $C_i$ of the subgraph $\Gamma\backslash A$ (in some arbitrary ordering of connected components).
Up to a prefactor, the Tutte polynomial is obtained from the universal $V$-function by the specialization $y_n=(x\!-\!1)(y\!-\!1)^n$.
 Tutte showed that the $V$-function is universal for graph invariants satisfying a deletion-contraction recurrence for non-loop edges, multiplicative over disjoint unions, and specified by boundary values on graphs all the edges of which are loops.
Examples of $V$-functions that are not a specialization of the Tutte polynomial have not been so widely studied, but see for example~\cite{WF11}.
\end{remark}

\subsection{Organization of the paper}
We begin in Section~\ref{sec:tutte_graphs} by recalling the definition of the Tutte polynomial of a graph and its specializations to the chromatic polynomial and the flow polynomial. In Section~\ref{sec:tutte_maps} we give the relevant background to orientably embedded graphs (maps). We then formally define the surface Tutte polynomial of a map and state our main results. We derive some basic properties of the surface Tutte polynomial and show how it specializes to the Bollob\'as--Riordan polynomial, the Krushkal polynomial and the Las Vergas polynomial of a map, as well as to the Tutte polynomial of the underlying graph.
In Section~\ref{sec:spec} we state the evaluations of the surface Tutte polynomial that give the number of nowhere-identity local flows (or tensions) taking values in any given finite group, leaving proofs to Section~\ref{sec:proofs}. 
We then derive further specializations of the surface Tutte polynomial, such as the number of quasi-forests.
In Section~\ref{sec:proofs} we enumerate nonabelian local flows and tensions. Finally, in Section~\ref{sec:conclusions} we review our results in the context of other work and identify some directions for future research.

%% file: tutte_graphs_maps_s2_s3_ini_120.tex
%!TEX root = /Users/crovellu/Dropbox/surface_tutte/First_surface_tutte/Stutte_I_120.tex
\section{The Tutte polynomial for graphs}\label{sec:tutte_graphs}
%\footnote{Lluis: recall the tutte polynomial for graphs in the introduction?}

%\section{The Tutte polynomial for graphs}

\subsection{Graphs}

A  graph $\Gamma=(V,E)$ is given by a set of vertices $V$ and a set of edges $E$, together with an incidence relation between vertices and edges such that any edge $e\in E$ is either incident to two different vertices $u,v\in V$ or is incident ``twice" to the same vertex $v\in V$. In the latter case $e$ is called a {\em loop}.  
If several edges are incident with the same pair of vertices $u,v$ then they are called {\em multiple edges}. 
The {\em degree} ${\rm deg}(v)$ of a vertex $v$ is the number of edges incident with it (any loop incident with the vertex is counted twice).

%\begin{definition}\label{def:del-con-graphs}
%Let $\Gamma=(V,E)$ be a graph and $e$ an edge of $\Gamma$. 

The graph $\Gamma\backslash e$ obtained from $\Gamma$ by {\em deletion} of $e$ is the graph $(V,E\setminus\{e\})$. 
The graph $\Gamma/e$ obtained from $\Gamma$ by {\em contraction} of $e$ is defined by first deleting $e$ and then identifying its endpoints. %incident with vertices $u,v$ is defined by deleting $e$ and then identifying $u$ and $v$ with a new single vertex $w$, which is defined to be adjacent to all vertices $x\in V\setminus\{u,v\}$ with the property that either $ux\in E$ or $vx\in E$. 
Contracting a loop of a graph coincides with deleting it.

\begin{definition}\label{def:graph_parameters}%[Notation for graph parameters]
For a graph $\Gamma$ we let $v(\Gamma)$, $e(\Gamma),$ $k(\Gamma)$ denote the number of vertices, edges and connected components of $\Gamma$. The {\em rank} of $\Gamma$ is defined by $$r(\Gamma)=v(\Gamma)-k(\Gamma),$$
and the {\em nullity} of $\Gamma$ by $$n(\Gamma)=e(\Gamma)-r(\Gamma)=e(\Gamma)-v(\Gamma)+k(\Gamma).$$
\end{definition}

An edge $e$ is a loop of $\Gamma$ precisely when $n(\Gamma/ e)=n(\Gamma)-1$; an edge $e$ is a {\em bridge} of  $\Gamma$ (deleting $e$ disconnects the connected component of $\Gamma$ to which it belongs) precisely when $r(\Gamma\backslash e)=r(\Gamma)-1$.

\subsection{The Tutte polynomial}\label{sec:Tutte_graph}

The Tutte polynomial of a graph $\Gamma=(V,E)$ is defined by the subgraph expansion
\begin{equation}\label{eq:tutte_e1}
T(\Gamma;x,y)=\sum_{A\subseteq E}(x-1)^{r(\Gamma)-r(\Gamma\backslash A^c)}(y-1)^{n(\Gamma\backslash A^c)},
\end{equation}
where $A^c=E\backslash A$ is the complement of $A\subseteq E$. 
% \item[(2)] Deletion-contraction recurrence:
% \begin{equation}\label{eq:tutte_rec}
% 	T(\Gamma;x,y)=\begin{cases} T(\Gamma/e;x,y)+T(\Gamma\backslash e;x,y) & \mbox{\rm if $e$ is ordinary,}\\
%  xT(\Gamma/e;x,y) &  \mbox{if $e$ is a bridge,}\\
%  yT(\Gamma\backslash e;x,y) &  \mbox{if $e$ is a loop,}\end{cases}
% \end{equation}
% while if $\Gamma$ is edgeless then $T(\Gamma;x,y)=1$.
% \end{itemize}

The Tutte polynomial contains the chromatic polynomial as a specialization. In particular,
the number of nowhere-zero $\mathbb Z_n$ tensions of $\Gamma$ is given by $(-1)^{r(\Gamma)}T(\Gamma;1-n,0)$.
 
Dually, the flow polynomial $\phi(\Gamma;z)$ evaluated at $n\in\mathbb N$ is equal to the number of nowhere-zero $\mathbb Z_n$-flows of $\Gamma$ and is given by $\phi(\Gamma;z)=(-1)^{n(\Gamma)}T(\Gamma;0,1-z).$

\section{A Tutte polynomial for maps}\label{sec:tutte_maps}

\subsection{Graph embeddings and maps}

For embeddings of graphs in surfaces we follow \cite{LZ04}. See also~\cite{EMM13}.

A {\em surface} in this paper is a compact oriented two-dimensional topological manifold. Such orientable surfaces are classified by a nonnegative integer parameter, called the {\em genus} $g$ of the surface (the number of ``handles'', or ``doughnut holes''); thus the sphere has genus $0$ and the torus genus $1$.
Surfaces are not only orientable but have been given a fixed orientation, which in particular allows one to distinguish left and right relative to a directed line.

\begin{definition}\label{def:map}
A \emph{connected map} $M$ is a graph $\Gamma$ embedded into a connected surface $\Sigma$ (that is, considered as a subset $\Gamma\subset \Sigma$) in such a way that
\begin{itemize}
\item[(1)] vertices are represented as distinct points in the surface
\item[(2)] edges are represented as continuous curves in the surface that intersect only at vertices 
\item[(3)] cutting the surface along the graph thus drawn, what remains, (that is, the set $\Sigma\setminus \Gamma$) is a disjoint union of connected components, called {\em faces}. Each face is homeomorphic to an open disk.
\end{itemize}
The graph $\Gamma$ is said to be the {\em underlying graph} of $M$.
\end{definition}
A map is also known as an {\em orientably embedded graph}, an {\em orientable ribbon graph}, %a {\em fat graph}, 
a {\em graph with a rotation system} or {\em cyclic graph},%~\cite{BR01}, 
with the attendant variations in diagrammatic representation of a map. (See~\cite{EMM13} and the references therein.)

A contractible closed curve in a surface $\Sigma$ is one that can be continuously deformed (or contracted) in $\Sigma$ to a single point. A cycle $C$ of a graph $\Gamma$ embedded in $\Sigma$ is contractible in $\Sigma$ if the subgraph $(V(C), E(C))$ of $\Gamma$ forms a contractible closed curve in $\Sigma$. A region of a surface is a 2-cell if its boundary is a contractible cycle. An embedding of a graph $\Gamma$ into a surface $\Sigma$ subject to the condition (3) in Definition~\ref{def:map} (that each of the connected components of $\Sigma\setminus \Gamma$ is homeomorphic to an open disk) is called a {\em 2-cell embedding} of $\Gamma$ and has the property that each face is a 2-cell. 
See Figure~\ref{fig:2cellembed}. 
A graph  must be connected in order for it to have a 2-cell embedding.

\ifpdf

\begin{figure}[htp]
\centering
\includegraphics[scale=1.3]{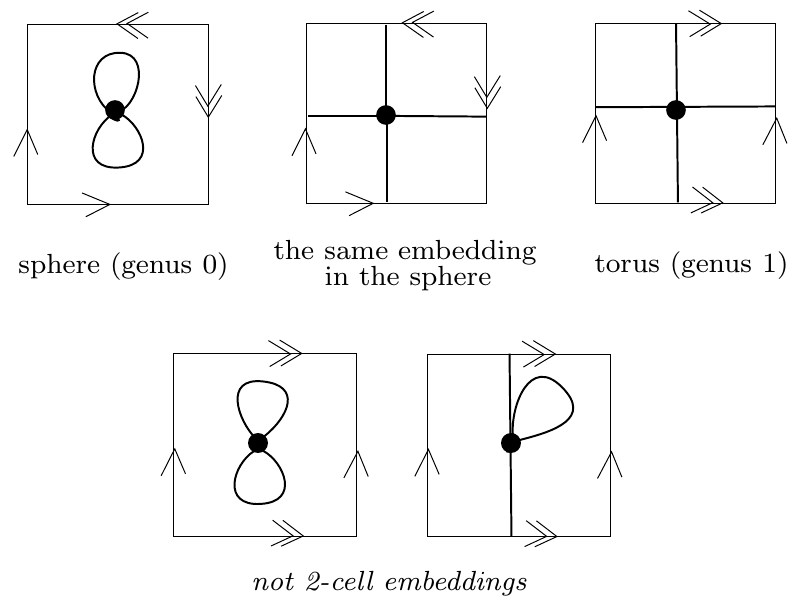}
\caption{Embeddings of the graph consisting of two loops on a single vertex, in the sphere and the torus. (Edges of the square with matching arrows are glued together.) The lower pair of embeddings are not 2-cell embeddings as in each of them one of the faces is not homeomorphic to an open disk. }\label{fig:2cellembed}
\end{figure}

\fi

\begin{definition}
	A {\em map} is a disjoint union of connected maps, one for each connected component of the underlying graph $\Gamma$: each connected component of $\Gamma$ is embedded as a connected map into its own surface. 
\end{definition}

\begin{definition}
Two maps $M_1\subset X_1$ and $M_2\subset X_2$, with underlying graphs $\Gamma_1$ and $\Gamma_2$, are {\em isomorphic} if there exists an orientation-preserving surface homeomorphism $u:X_1\to X_2$ such that the restriction of $u$ on $\Gamma_1$ is a graph isomorphism between the underlying graphs $\Gamma_1$ and $\Gamma_2$.
\end{definition}

A map parameter $P=P(M)$ is said to be a {\em map invariant} if $P(M_1)=P(M_2)$ whenever $M_1$ and $M_2$ are isomorphic.

\begin{definition}
The {\em genus} of a connected map $M$ is the genus of the connected surface $\Sigma$ in which it is embedded. 
 The genus of a graph is the minimum genus of an embedding of $\Gamma$ as a map. 
\end{definition}

In particular, graphs of genus $0$ are called {\em planar}, the maps witnessing this being {\em plane graphs} (or plane maps). Plane maps are viewed rather as embeddings of planar graphs in the sphere 
(the unbounded outer face in the plane becomes bounded once the plane has a point at infinity added to make it a sphere).
For a given graph $\Gamma$ in general there exist non-isomorphic maps, and these maps may be of various genera. For example, the tetrahedron $K_4$ is usually seen as a plane map, equal to the skeleton of the tetrahedral polytope in which every face is a triangle, but also has a genus $1$ embedding in which there are only two faces, one of degree $4$ and the other of degree $8$, and another genus $1$ embedding with one face of degree $3$ and one face of degree $9$. (The degree of a face is the number of adjacent edges to the face, one-faced edge being counted twice.)

To a map $M$ embedding a graph $\Gamma=(V,E)$ in a surface $\Sigma$ we identify the vertices and edges of $\Gamma$ with their representations in $\Sigma$ 
 and let $F$ be the set of faces of $M$. A face is identified with the subset of edges forming  its boundary.

% For a connected map let $\chi(M)=|V|-|E|+|F|$ denote its Euler characteristic.
% %\end{definition}
% The well-known formula of Euler is that $\chi(M)=2-2g(M)$ for a connected map $M$.

For a connected map $M=(V,E,F)$ let $v(M)=|V|$, $e(M)=|E|$, $f(M)=|F|$, and let $\chi(M)=|V|-|E|+|F|$ be its Euler characteristic.
%\end{definition}
The well-known formula of Euler is that $\chi(M)=2-2g(M)$ for a connected map $M$, where $g(M)$ denotes the genus of $M$.

We extend these map parameters additively over disjoint unions: 
%The parameters $f$,$g$ and $\chi$ are defined for disconnected maps are extended additively over the disjoint union of the connected components of the map.

\begin{definition}%[Notation for map parameters]
Let $M=M_1\sqcup \cdots \sqcup M_{k}$ be a map, equal to the disjoint union of connected maps $M_1,\dots, M_k$. %a disconnected map embedding each connected component of a graph $\Gamma=(V,E)$ into a surface. Let 
We define $k(M)=k$ (the number of connected components of the underlying graph of $M$) and set %embedding one of the $k=k(M)$ connected components of $\Gamma$, %and $A\subseteq E$,
$$v(M)=\sum_{i=1}^{k} v(M_i),\: e(M)=\sum_{i=1}^{k} e(M_i), \: f(M)=\sum_{i=1}^{k} f(M_i),$$
$$g(M)=\sum_{i=1}^{k} g(M_i),
\text{ and }\; \chi(M)=\sum_{i=1}^{k} \chi(M_i).
$$ %denote the number of faces in the map $M_A$ and $f(M/A)$ the number of faces in the map $M/A$.

The {\em rank} and {\em nullity} of $M$ are defined by
$$r(M)=v(M)-k(M), \quad n(M)=e(M)-v(M)+k(M),$$
and the {\em dual rank} and {\em dual nullity} by
$$r^*(M)=f(M)-k(M), \quad n^*(M)=e(M)-f(M)+k(M).$$
\end{definition}

The number of vertices, edges and connected components and the rank and nullity of a map are parameters shared with those of its underlying graph (Definition~\ref{def:graph_parameters}) and we use the same notation for them.

\begin{remark}
The number of faces $f(M)$ in a map is the sum of the number of faces in the embeddings of components of $M$ in disjoint surfaces: embedding a disconnected graph in one surface does not give a map (as there is a face not homeomorphic to an open disk). %Often a planar graph $\Gamma$ with $k(\Gamma)$ connected components is pictured as embedded in one plane/sphere, and the total number of faces counted as interior faces plus one ``outer face": this differs by $k(\Gamma)-1$ from the count given by $f(M)$ when embedding connected components in separate spheres as maps $M_1,\dots M_k$ that together form the disconnected map~$M$.
For a connected map $M$, $g(M)$ is the genus of the surface defined by the map, while for disconnected maps $g(M)$ is the genus of the surface in which all the components can be simultaneously embedded in which all but one face is a 2-cell (the face incident with each connected component is homeomorphic to a disk with $k(M)-1$ holes in its interior). 

%In~\cite[Sect. 3.1.3]{CMNR16} it is unclear what is meant by the genus of the surface formed by capping holes of a ribbon graph in order to form a surface when this graph is disconnected. It appears not to be the sum of the genera of components (which is the usual convention, which we adopt) as Euler's formula involves a corrective term of number of connected components, whereas our formula does not require this term. 
\end{remark}
%Since the number of vertices $v(M)$ and number of edges $e(M)$ are both additive over disjoint unions of maps, we have that for  a (possibly disconnected) map $M=(V,E,F)=M_1\sqcup \cdots \sqcup M_k$ embedding a graph $\Gamma=(V,E)$, where each $M_j$ is a connected map, 
By applying Euler's formula for connected maps to each connected component of a disconnected map $M$ and using the additivity of the parameters $v,e,f$ and $g$ over disjoint unions we have 
\begin{equation} \label{eq:euler_rel}
\chi(M)=v(M)-e(M)+f(M)=2k(M)-2g(M).
\end{equation}
This in turn implies 
\begin{equation}\label{eq:dual_nullity}n^*(M)=r(M)+2g(M).\end{equation}
 % (while there are $k(M)$ times the constant $2$).

\begin{definition}\label{def:qt-bouq}
A map $M$ is a \emph{quasi-tree} if $f(M)=1$ and a \emph{bouquet} if $v(M)=1$.
\end{definition}

Given a map $M$, $M^{\ast}$ denotes its (surface) dual map, the vertices of which are the faces of $M$ and edges of $M^*$ join adjacent faces of $M$. For instance, if one edge $e$ is adjacent to only one face $f$, then the vertex corresponding to $f$ contains a loop corresponding to $e$.
% See Definition~\ref{def:dualmap} for a precise definition. %The dual of a quasi-tree is a bouquet. 
The map $M^{\ast}$ lies in the same surface as $M$, so that $g(M^{\ast})=g(M)$. We have $r^*(M)=r(M^*)$ and $n^*(M)=n(M^*)$.

Given an orientably embedded connected graph, the orientation of the surface defines, for any vertex $v$, a cyclic rotation of the edges incident with $v$ (take the edges in the anticlockwise cyclic order with respect to the orientation of the surface). In order to distinguish both ends of the edges, we can assign an arbitrary direction to each undirected edge (loop or non-loop). Furthermore, this allows us to recover each face of the embedded graph by a sequence of the following two-step process starting with a vertex and an edge attached to that vertex: %\footnote{Guus: also say that the converse holds. Any rotation system induces an embedding in a surface.}
\begin{itemize}
\item move to the other end (the vertex attached to the other end of the edge)
\item select the previous edge in the anticlockwise cyclic order given by the orientation.
\end{itemize}
The process finishes when we are at the original vertex and we are about to repeat the same edge in the same direction.

 Conversely, every rotation scheme defines a unique 2-cell embedding of a connected graph on a closed oriented surface (up to isomorphism). For more on rotation systems see~\cite[page~36]{LZ04}.% \footnote{Lluis: add citation}

Given a map $M=(V,E,F)$ and $e\in E$, the map obtained by {\em deleting} $e$ is denoted by $M\backslash e$ and the map obtained by {\em contracting} $e$ by $M/e$. For $A\subseteq E$ we let $M\backslash A$ ($M/A$) denote the map obtained by deleting (contracting) all the edges in $A$, the order in which  the edges in $A$ are taken being immaterial. A {\em submap} of $M$ is a map of the form $M\backslash A$ and shares the same vertex set as $M$.  The operations of deletion and contraction have standard definitions in terms of the ribbon graph representation of maps, see for example~\cite{CMNR16, EMM13}. Informally, the deletion of $e$ is defined by removing the two ends of the edge $e$ from the neighbourhoods of the vertices where $e$ is attached to, while maintaining the same rotation on the remaining edges of the vertex (moving to the next undeleted edge incident with the vertex in the original rotation). Contraction can be defined by duality: $(M/e)^{\ast}=M^{\ast}\setminus e^{\ast}$. 
We have $M\backslash e/f\cong M/f\backslash e$ for distinct edges $e$ and $f$. 

Deletion of an edge $e$ in a map $M$ corresponds to deleting the edge in its underlying graph $\Gamma$, although deletion of $e$ may reduce the genus of the map, in which case the underlying graph $\Gamma\backslash e$ of $M\backslash e$ is embedded in a different surface to $\Gamma$.  We have $v(M)=v(M\backslash e)=v(\Gamma\backslash e)=v(\Gamma)$, $e(M)-1=e(M\backslash e)=e(\Gamma\backslash e)=e(\Gamma)-1$ and $k(M\backslash e)=k(\Gamma\backslash e)$.
In particular,   
\begin{equation}\label{eq:mullity_del}n(M\backslash e)=n(\Gamma\backslash e).\end{equation}

Contraction of an edge in a map does not always correspond to contraction of the edge in the underlying graph. In particular, contracting a loop in a map has the effect of splitting in two the vertex with which it is incident. %The behaviour of rank and nullity for maps under contraction differs from the behaviour of their counterparts for graphs.  %Deleting edges (loops and non-loops) in a plane map is equivalent to deleting edges in the underlying graph.

%In Section~\ref{sec:map_del_con} we give a definition of deletion and contraction in terms of combinatorial maps (Definitions~\ref{def:del} and~\ref{def:con}) %, see also~\cite{LZ04}) 
%which is more convenient in the context of counting local flows and tensions of maps. 

%\begin{remark}\label{rem:gen_no_drop}
%It is not difficult to see that, by deleting and contracting, the genus of map does not increase.
%\end{remark}

\begin{lemma}\label{lem:genus_submaps}
If $M$ is a map and $A$ a subset of edges then 
$$g(M\backslash A^c)+g(M/A)\leq g(M),$$
with equality if and only if
% \[
% k(M\backslash A^c)\!\!-k(M)\!-\!f(M\setminus A^c)\!\!+k(M/A)=0 \text{ and }k(M/A)\!-k(M)\!-\!v(M/A)+\!k(M\backslash A^c)\!=0.
% \]
\[
k(M\backslash A^c) -k(M) - f(M\setminus A^c) +k(M/A)=0 \text{ and }k(M/A)-k(M)-v(M/A)+k(M\backslash A^c)=0.
\]
\end{lemma}
\begin{proof} Using Euler's relation,
\begin{align*}
2g(M\backslash A^c)\!+\!2g(M\!/\!A) & = 2k(M\backslash A^c)\!+\!2k(M\! /\! A)\!-\!v(M\backslash A^c)\!-\!v(M\! /\! A)\!+\!e(M\backslash A^c)\!+\!e(M\! /\! A)\\
& \quad\quad\quad\quad\quad-f(M\backslash A^c)\!-\!f(M\! /\! A)\\
& = 2g(M)+[k(M\backslash A^c)\!\!-k(M)\!-\!f(M\setminus A^c)\!\!+k(M\! /\! A)]\!
\\
& \quad\quad\quad\quad\quad +\![k(M\! /\! A)\!-k(M)\!-\!v(M\! /\! A)+\!k(M\backslash A^c)]
\end{align*}
We now claim that
\begin{equation}
k(M\! /\! A)\!-k(M)\!-\!v(M\! /\! A)+\!k(M\backslash A^c)\leq 0.\label{eq:components}
\end{equation}
%from which the inequality $k(M\backslash A^c)\!\!-k(M)\!-\!f(M\setminus A^c)\!\!+k(M/A)\leq 0$ follows by duality.
\input{alternative_key_lemma_120.tex}

%Indeed, the number of components of $M\backslash A^c$ is always at most the number of vertices of $M/A$.
%If $k(M/A)=k(M)$ then we are done. So let us suppose that $k(M/A)>k(M)$.
%Each additional component that is created when contracting $A$ must give rise to newly created vertices that do not belong to any of the components of $M\backslash A^c$. 
%This implies that $k(M/A)-k(M)$ is at most $v(M/A)-k(M\backslash A^c)$ and hence proves \eqref{eq:components}.
By duality we immediately obtain $k(M\backslash A^c)\!\!-k(M)\!-\!f(M\setminus A^c)\!+\!k(M/A)\leq 0$.
This proves the lemma.
\end{proof}
A simple but useful corollary of Lemma~\ref{lem:genus_submaps} is that neither deletion nor contraction of edges increase the genus: $g(M\backslash A^c)\leq g(M)$ and $g(M/A)\leq g(M)$. %: $g(M\backslash e)$ and $g(M/e)$ are bounded above by $g(M)$.

\subsection{A Tutte polynomial for maps}

% In this section we follow Tutte in how he defined the dichromate of a graph as a simulatenous generalization of the chromatic and flow polynomials and define a polynomial invariant for maps that includes the number of nowhere-identity local $G$-tensions and the number of nowhere-identity local $G$-flows as specializations.
% \subsection{Submap expansion}

\begin{definition} \label{def:surface_tutte}
Let $\mathbf x=(x, x_0,x_1,x_2,\dots)$, $\mathbf y=(y, y_0,y_1,\dots)$ be two infinite sequences of commuting indeterminates.

Given a map $M=(V,E,F)$, the {\em surface Tutte polynomial} of $M$ is the multivariate polynomial
\begin{equation}\label{eq:surface_Tutte}\mathcal T(M;\mathbf x,\mathbf y)=\sum_{A\subseteq E}x^{n^*(M/A)}y^{n(M\backslash A^c)}\prod_{\stackrel{\mbox{\rm \tiny conn. cpts $M_i$}}{\mbox{\rm \tiny of $M/A$}}}x_{g(M_i)}\prod_{\stackrel{\mbox{\rm \tiny conn. cpts $M_j$}}{\mbox{\rm \tiny of $M\backslash A^c$}}}y_{g(M_j)},\end{equation}
where $A^c=E\setminus A$ for $A\subseteq E$.  %, $M^*$ denotes the geometric dual of the map $M$, 
%and $n(M)=e(M)-v(M)+k(M)$ denotes the nullity of $M$.
\end{definition}
\begin{remark} \label{rmk:finitely_many_var}
By Lemma~\ref{lem:genus_submaps}, for a given map $M$, $\mathcal T(M;\mathbf x,\mathbf y)$ is a polynomial in indeterminates $x,x_0,\dots, x_{g(M)}$ and $y, y_0,\dots, y_{g(M)}$. 
\end{remark}
%\begin{remark}, 
\begin{remark}\label{rmk:alt_def}
In the summation~\eqref{eq:surface_Tutte} defining $\mathcal T(M;\mathbf x, \mathbf y)$ the exponent of $x$ is % (Lemma~\ref{lem:rank_geom_dual}) as 
$n^*(M/A)=n((M/A)^*)=n(M^*\backslash A)$.
%=n(M^*\backslash A)$,
(For convenience, and where no confusion can arise, we make the usual identification of the edges of the surface dual $M^*$ with the edges of $M$.) %, writing $M^*\backslash e$ and $M^{\ast}/e$ instead of $M^{\ast}/e^{\ast}$.) %dual of $M=(V,E,F)$ with $M^*=(F,E,V)$.

 %, and the exponent of $y$ is $$e(M\backslash A^c)-v(M\backslash A^c)+k(M\backslash A^c).$$ 
%The quantity $v(M\backslash A^c)$ is constant, equal to $|V|$, while the dual number $f(M/A)$ varies. %[ - perhaps we should rather have $v(M/A)$ and $f(M\backslash A^c)$, where there is dual variation?]
%Since $g(M^*)=g(M)$ and the connected components of $M/A$ are the duals of the connected components of $M^*\backslash A$, we may alternatively make the definition:
%$$\prod_{\stackrel{\mbox{\rm \tiny conn. cpts $M_i$}}{\mbox{\rm \tiny of $M/A$}}}x_{g(M_i)}=\prod_{\stackrel{\mbox{\rm \tiny conn. cpts $M_i$}}{\mbox{\rm \tiny of $M^*\backslash A$}}}x_{g(M_i)}.$$
%\begin{equation}\label{eq:surface_Tutte_dual}\mathcal T(M;\mathbf x,\mathbf y)=\sum_{A\subseteq E}x^{n(M^*\backslash A)}y^{n(M\backslash A^c)}\prod_{\stackrel{\mbox{\rm \tiny conn. cpts $M_i^*$}}{\mbox{\rm \tiny of $M^*\backslash A$}}}x_{g(M_i^*)}\prod_{\stackrel{\mbox{\rm \tiny conn. cpts $M_j$}}{\mbox{\rm \tiny of $M\backslash A^c$}}}y_{g(M_j)}.\end{equation}
\end{remark}

%%%%%%%%%%%%%%%%%%%%%%%%%%%%%%%%%%%%%%%%%%%%%%%%%%%%%%%%%%%%%%%%%%%%%%%
%%%%%%   Remark about V-function and the Tutte polynomial specialization %%%%%%%%%%%
%%%%%%%%%%%%%%%%%%%%%%%%%%%%%%%%%%%%%%%%%%%%%%%%%%%%%%%%%%%%%%%%%%%%%%%

% \begin{remark}
% 	In light of Remark~\ref{rmk:Tutte_V-function}, and as Proposition~\ref{prop:specs} shows, the term
% 	\[\sum_{A\subseteq E}y^{n(M\setminus A^{c})}
% 	 \prod_{\stackrel{\mbox{\rm \tiny conn. cpts $M_j$}}{\mbox{\rm \tiny of $M\backslash A^c$}}}y_{g(M_j)}\]
% 	 specializes to the Tutte polynomial for graphs.
% \end{remark}

% By Lemma~\ref{lem:rank_geom_dual}, in the summation~\eqref{eq:surface_Tutte} defining $\mathcal T(M;\mathbf x, \mathbf y)$ the exponent of $x$ may alternatively be written as 
% \begin{align*}e(M/A)-f(M/A)+k(M/A) & =r(M/A)+2g(M/A)\\
% & =n((M/A)^*)=n(M^*\backslash A),\end{align*}

See Figure~\ref{fig:TwoLoopsTorus} for a small example of the calculations involved in computing $\mathcal T(M;\mathbf x,\mathbf y)$ from its subset expansion~\eqref{eq:surface_Tutte}.

\ifpdf

\begin{figure}
\centering
\begin{minipage}{0.48\textwidth}
\includegraphics[scale=1.05]{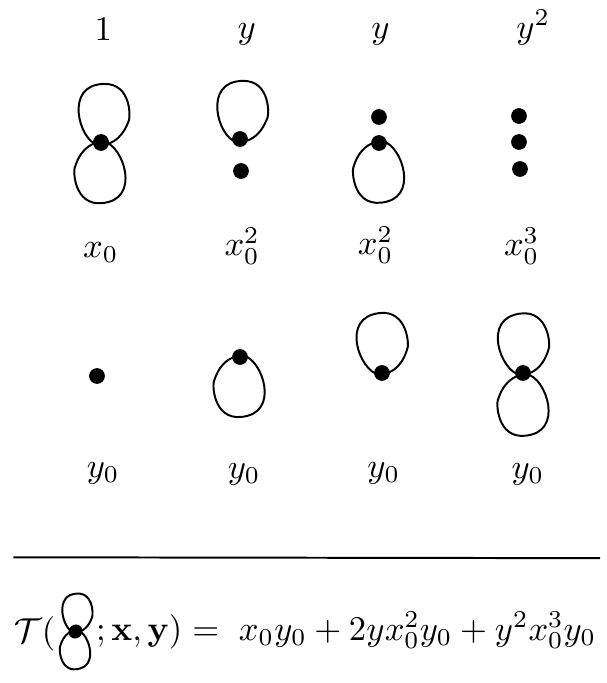}
%\caption{The surface Tutte polynomial of plane map consisting of two loops on a single vertex.}\label{fig:TwoLoops}
%\end{figure}
\end{minipage}
\hfill
%\begin{figure}
%\centering
\begin{minipage}{0.48\textwidth}
\includegraphics[scale=1.05]{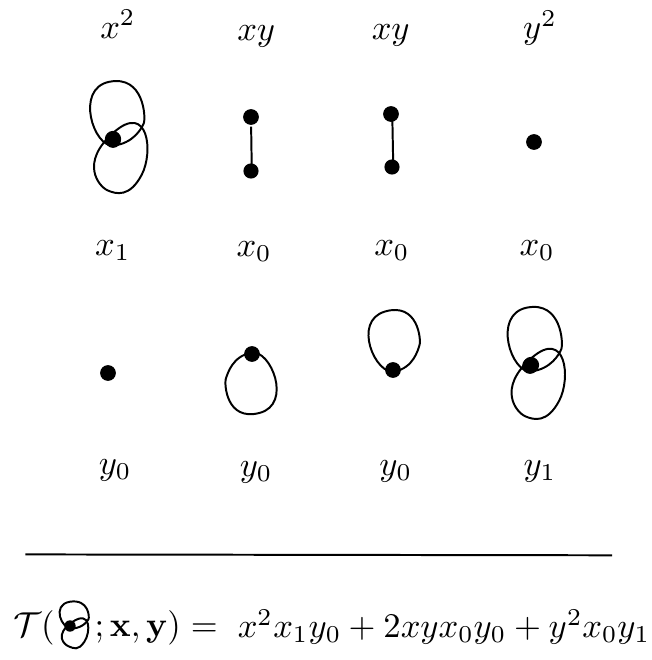}
-\end{minipage}
\caption{The surface Tutte polynomial of two loops on a single vertex, on the left embedded in the plane and on the right embedded in the torus. % (two interlacing one-faced loops).
}\label{fig:TwoLoopsTorus}
\end{figure}
\else
\fi

The surface Tutte polynomial is multiplicative over the connected components of a~map:
\begin{proposition}\label{prop:multiplicativity}
For maps $M_1$ and $M_2$,
$$\mathcal T(M_1\sqcup M_2;\mathbf x,\mathbf y)=\mathcal T(M_1;\mathbf x,\mathbf y)\mathcal T(M_2;\mathbf x,\mathbf y).$$ 
\end{proposition}
\begin{proof} The nullity parameter $n(M)$ is additive over disjoint unions. The set of connected components of $(M_1\sqcup M_2)/A$ is the disjoint union of the connected components of $M_1/A_1$ and those of $M_2/A_2$, where $A=A_1\sqcup A_2$ with $A_1\subseteq E(M_1)$ and $A_2\subseteq E(M_2)$. Likewise for $(M_1\sqcup M_2)\backslash A^c$.\end{proof}

The surface Tutte polynomial behaves with respect to geometric duality in the same way as the Tutte polynomial of a matroid does with respect to matroid duality:
\begin{proposition}\label{prop:duality}
If $M$ is a map and $M^*$ its surface dual then
$$\mathcal T(M^*;\mathbf x,\mathbf y)=\mathcal T(M;\mathbf y,\mathbf x).$$
\end{proposition}
\begin{proof} This is immediate from Definition~\ref{def:surface_tutte} and the fact that $n^*(M^*/A)=n(M\backslash A)$, $n(M^*\backslash A^c)=n^*(M/A^c)$, $g(M^*)=g(M)$ and that the connected components of $M^*/A$ (respectively $M^*\backslash A^c$) are in one-one correspondence with and have the same genus as the connected components of its dual $M\backslash A$ (respectively $M/A^c$). 
\end{proof}

The surface Tutte polynomial coincides with the Tutte polynomial for plane embeddings of planar graphs:

\begin{proposition}\label{prop:surface_Tutte_plane}
If $M$ is a plane map embedding of a planar graph $\Gamma$ then $\mathcal T(M;\mathbf x,\mathbf y)$ is a polynomial in $x,y,x_0,y_0$ and $$\mathcal T(M;\mathbf x,\mathbf y)=(x_0y_0)^{k(\Gamma)}T(\Gamma;y_0x+1,x_0y+1).$$
\end{proposition}
\input{proof_plane_101.tex}
%\begin{proof}
%For a plane map $M$ we have $g(M)=0$ and every spanning submap also has genus $0$. In this case $\mathcal{T}(M;\mathbf x,\mathbf y)$ is a polynomial in $x,y,x_0,y_0$.
%Using equation~\eqref{eq:dual_nullity} and the defining submap expansion~\eqref{def:surface_tutte}, we have
%$$\mathcal{T}(M;\mathbf x,\mathbf y)=\sum_{A\subseteq E}x^{r(M/A)}y^{n(M\backslash A^c)}x_0^{k(M/A)}y_0^{k(M\backslash A^c)}.$$
% Since $r(M/A)=r(M)-r(M\backslash A^c)$ and $k(M/A)=k(M)-k(M\backslash A^c)$ [....]
%
%\end{proof}

\subsection{Relation to other map polynomials}\label{sec:relation}

%\begin{remark}
The Krushkal polynomial~\cite{K11} of a graph $\Gamma$ embedded in a surface $\Sigma$ as a map $M$ is, using the definition in~\cite{CMNR16} and our notation, given by %\footnote{In Krushkal's original paper the definition given is
%\begin{equation}\label{eq:Krushkal}\mathcal K(M;x,y,a,b)=\sum_{A\subseteq E}(x-1)^{r(\Gamma)-r(\Gamma\setminus A^c)}y^{n(\Gamma\setminus A^c)-g(M)-g(M\setminus A)+g(M/A)}a^{g(M/A)}b^{g(M\setminus A^c)},\end{equation}
%where the exponent of $y$ is in fact a more straightforwardly defined topological parameter (in terms homology groups) than it appears. This is a different normalization of the polynomial, which Krushkal chooses as it satisfies the duality relationship $K(M^*;x,y,a,b)=K(M;y,x,b,a)$ more tidily than the version with nullity in the exponent of $y$. }
$$\mathcal K(M;x,y,a,b)=\sum_{A\subseteq E}(x-1)^{k(M\setminus A^c)-k(M)}y^{n(M\backslash A^c)}a^{g(M/A)}b^{g(M\setminus A^c)}.$$
%where the matroid rank is used for $\Gamma$ (as defined for graphs in Definition~\ref{def:graph_parameters}). %Note that $r(\Gamma^*)-r(\Gamma^*/A)=n(\Gamma\setminus A^c)$ is the nullity of the subgraph $(V,A)$. 

The Bollob\'as--Riordan polynomial~\cite{BR01} of a graph $\Gamma$ orientably embedded in $\Sigma$ as a map $M$ is defined by
$$\mathcal R(M;x,y,z)=\sum_{A\subseteq E}(x-1)^{k(M\setminus A^c)-k(M)}y^{n(M\setminus A^c)}z^{2g(M\setminus A^c)}$$
and is obtained from the Krushkal polynomial~\cite{K11} by setting $a=1, b=z^2$ in $\mathcal K(M;x,y,a,b)$. (The Krushkal polynomial and Bollob\'as--Riordan polynomial are more generally defined for embeddings of graphs in non-orientable surfaces~\cite{BR02, K11}.) % Generalizing the surface Tutte polynomial to graphs embedded in non-orientable surfaces awaits future work. )
% and its deletion-contraction minors and the map rank $v(M)-k(M)$ for $M$ and its deletion-contraction minors. (As we shall see, edge deletion and contraction for graphs and maps do not always correspond.) 
% and we have used $g(M/A)=g((M^*\backslash A)^*)=g(M^*\backslash A)$ in order to rewrite the exponent of $b$ to bring out more clearly the relationship with our surface Tutte polynomial. However, 

%For plane maps $M$, $\mathcal Q(M;x,y,a,b)$ specializes to the Krushkal polynomial (and vice versa), when they both coincide with the Tutte polynomial of the underlying graph of $M$ with the appropriate choice of variables. 

%For general maps $M$, while $\mathcal Q(M;x,y,a,b)$ no longer specializes to the Krushkal polynomial, 

The Las Vergnas polynomial~\cite{LV80} of a map $M$ is shown in~\cite{ACEMS13}, \cite[Prop. 3.3]{EMM15} to be given by, in our notation, 
$$\mathcal L(M;x,y,z)=\sum_{A\subseteq E}(x\!-\!1)^{k(M\backslash A^c)-k(M)}(y\!-\!1)^{n(M\backslash A^c)-g(M)-g(M\backslash A^c)+g(M/A)}z^{g(M)-g(M\backslash A^c)+g(M/A)}.$$

The surface Tutte polynomial $\mathcal T(M;\mathbf x,\mathbf y)$ specializes to the Krushkal polynomial (and hence the Bollob\'as--Riordan polynomial of $M$, the Las Vergnas polynomial of $M$, and the Tutte polynomial of the underyling graph of $M$), as may be verified by making the requisite substitutions:
%We moreover note that setting $a=b=1$ in the Krushkal polynomial we obtain the Tutte polynomial of the underlying graph (up to an obvious change of coordinates). We note this is not the case for our polynomials $\mathcal{Q}$ and $\widetilde{\mathcal{Q}}$. See for example Propositions~\ref{prop:plane_quasi-trees_bouquets} and \ref{prop:quasi-trees_bouquets} below. We refer the reader to Section~\ref{sec:not_Kruskal} for more details on how the surface Tutte polynomial relates to the Krushkal polynomial and the original Tutte polynomial.
%\end{remark}
%\begin{remark}
%The Bollob\'as--Riordan polynomial~\cite{BR01} of a graph $\Gamma$ orientably embedded as a map $M$ is defined by
%$$\mathcal R(M;x,y,b)=\sum_{A\subseteq E}(x-1)^{r(\Gamma)-r(\Gamma\setminus A^c)}y^{n(\Gamma\setminus A^c)}b^{g(M\setminus A^c)}.$$
%This is a specialization of the Krushkal polynomial~\cite{K11}, setting $a=1$ in $\mathcal K(M;x,y,a,b)$. % and hence of $\mathcal Q(M;x,y,a,b)$ for plane $M$ by setting $a=1$.
%The Bollob\'as--Riordan polynomial is more generally defined for embeddings of graphs in non-orientable surfaces~\cite{BR02}. % Generalizing the surface Tutte polynomial to graphs embedded in non-orientable surfaces awaits future work. 
%\end{remark}

\begin{proposition}\label{prop:specs}
The surface Tutte polynomial $T(M;\mathbf x,\mathbf y)$ in indeterminates $\mathbf x=(x,x_0,$ $x_1,x_2,\dots)$ and $\mathbf y=(y,y_0,y_1,y_2,\dots)$ has the following specializations:
\begin{itemize}
\item the Krushkal polynomial of a map $M$ is given by 
$$\mathcal K(M;X,Y,A,B)=(X-1)^{-k(M)}\mathcal T(M;\mathbf x,\mathbf y),$$
in which $x=1, x_g=A^g, y=Y, y_g=(X-1)B^g$ for $g=0,1,2,\dots$, 
\item the Bollob\'as--Riordan polynomial  of a map $M$ is given by
$$\mathcal R(M;X,Y,Z)=(X-1)^{-k(M)}\mathcal T(M;\mathbf x,\mathbf y),$$ 
in which $x=1=x_g$, $y=Y$ and $y_g=(X-1)Z^{2g}$ for $g=0,1,2,\dots$,%\footnote{A: cf. number of nowhere-identity $G$-flows, which also has $x=1=x_g$.}
\item the Las Vergnas polynomial  of a map $M$ is given by 
$$\mathcal L(M;X,Y,Z)=(X-\!1)^{-k(M)}(Y-\!1)^{-g(M)}Z^{g(M)}\mathcal T(M;\mathbf x,\mathbf y),$$ 
with $x=1$, $x_g=(Y-\!1)^gZ^g$, $y=Y-\!1$, $y_g=(X-\!1)(Y-\!1)^{-g}Z^{-g}$ for $g=0,1,2,\dots$,  
and 
\item the Tutte polynomial of a graph $\Gamma$ is given by
$$T(\Gamma;X,Y)=(X-\!1)^{-k(M)}\mathcal T(M;\mathbf x,\mathbf y),$$ 
with $x=1=x_g$, $y=Y-\!1$, $y_g=X-\!1$ for $g=0,1,2,\dots$, and in which $M$ is an arbitrary embedding of $\Gamma$ as a map. 
\end{itemize}

\end{proposition}
%\begin{proof} 
  %\end{proof}
Figure~\ref{fig:MapPolynomials} displays the relationship between the various map polynomials in Proposition~\ref{prop:specs} together with the polynomial $\mathcal Q(M;X,Y,A,B)$, defined at the end of Section~\ref{sec:intro_surface_Tutte} and which is considered in more detail in Section~\ref{sec:specializations}. 
%; 
The relationship between $\mathcal Q(M;X,Y,A,B)$ and the Krushkal polynomial %, and between trivariate and bivariate specializations of $\mathcal Q(M;X,Y,A,B)$ and the Bollob\'as--Riordan polynomial, Las Vergnas polynomial and Tutte polynomial,  
is as yet unclear. In Section~\ref{sec:not_Kruskal} we ask whether  $\mathcal Q(M;X,Y,A,B)$ and the Krushkal polynomial are equivalent as map invariants  (Problem~\ref{prob:QK}), even though neither one appears to be a specialization of the other. 

\begin{figure}
\centering
\includegraphics[scale=0.96]{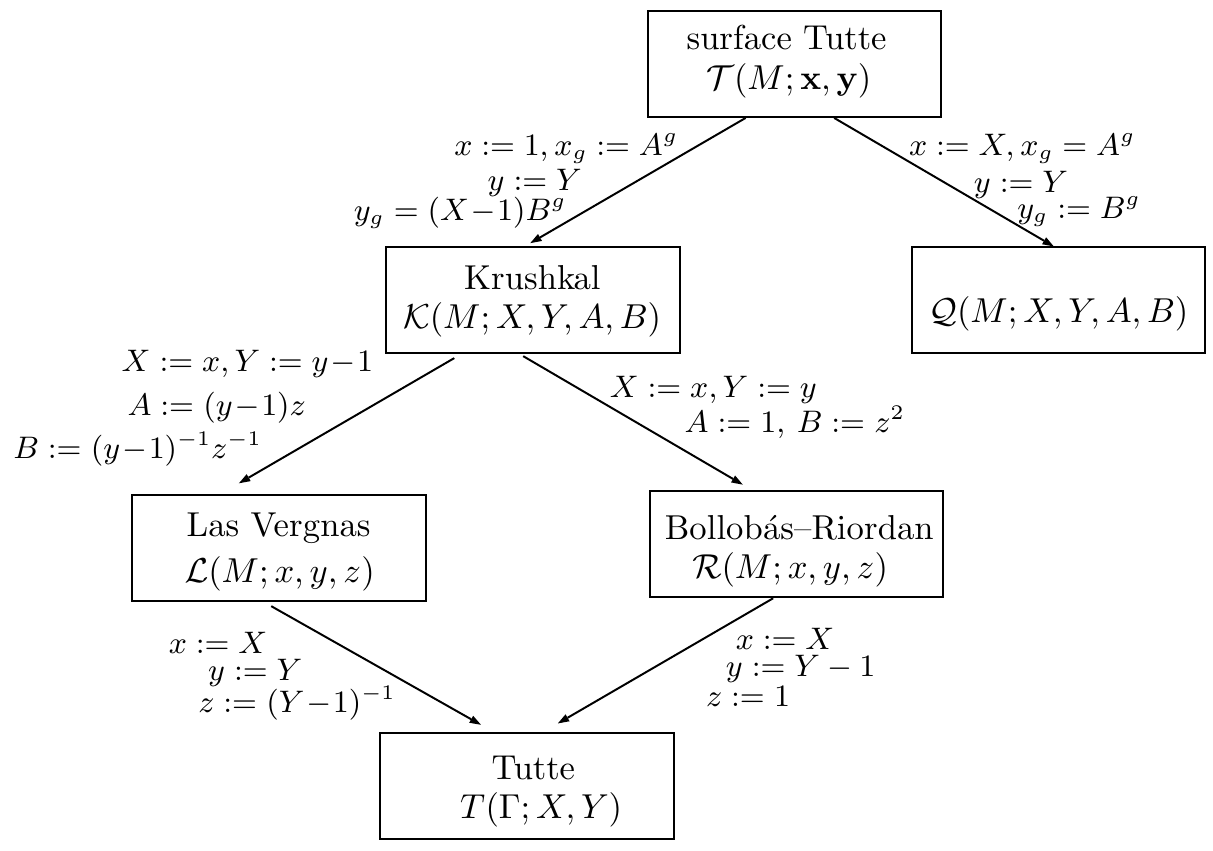}
\caption{Specializations of the surface Tutte polynomial for a graph $\Gamma$ embedded in an orientable surface $\Sigma$ as a map $M$ (prefactors, dependent only on $k(M)$ and $g(M)$, have been omitted -- see Proposition~\ref{prop:specs}.) When $M$ is plane ($\Sigma$ is the sphere) the hierarchy collapses to the Tutte polynomial of the underlying planar graph $\Gamma$. }\label{fig:MapPolynomials}
\end{figure}

%% file: alternative_key_lemma_120.tex
%\begin{claim}
%	$k(M/A)\!-k(M)\!\leq \!v(M/A)-\!k(M\backslash A^c)$
%\end{claim}

%\begin{proof}[Proof of the claim]
	
It suffices to prove that $k(M/A)\!-1\!\leq \!v(M/A)-\!k(M\backslash A^c)$ for a connected map $M$, since the map parameters $v(M)$ and $k(M)$ are additive over disjoint unions. %Let $M=\cup_{i=1}^{k(M)} M_i$, where each $M_i$ is a connected map. Consider $A=\cup_{i\in [k(M)]} A_i$, $A_i\subseteq E(M_i)$. 

Let $F\subseteq A$ be a maximal spanning forest of $M\backslash A^c$ (that is, $M\backslash F^c$ contains no cycles and for each $e\in A\backslash F$ the underlying graph of  $M\backslash (F\cup\{e\})^c$ contains a cycle). 
%Since $F$ is a maximal spanning forest of $M\backslash A^c$ 
We have then $k(M\backslash F^c)=k(M\backslash A^c)$ and we now need to prove that $k(M/A)-1\leq v(M/A)-k(M\backslash F^c)$. %$k(M/F)-1=\!v(M/F)-\!k(M\backslash A^c)$, 

Each edge in $F$ is a non-loop of $M$. Contracting a non-loop edge in $M$ corresponds to its contraction in the underlying graph $\Gamma$, and in particular preserves connectivity of $M$ (the edge is deleted, its endpoints $u$ and $v$ are fused into one vertex, whose incident edges are the other edges incident with $u$ and $v$ taken in the cyclic order inherited from the vertex rotations around $u$ and $v$).
We thus have $k(M/F)\!-\!1=1-\!1=0= \!v(M/F)-\!k(M\backslash F^c)$, as each connected component of $M\backslash F^c$ is reduced to a single vertex in~$M/F$. 
 %It now suffices to show that $k(M/A)\!-1\!\leq \!v(M/A)$.
	
%If $A_i$ contains a cycle then let $F_i$ be a maximal spanning forest of $A_i$ and $C_i=A_i\setminus F_i$. 
If $A=F$, we thus have equality in~\eqref{eq:components}. Suppose then that there is $e\in A\backslash F$. There is a unique cycle of $M\backslash A^c$ whose edges are contained in $F\cup\{e\}$. % Each edge in $C_i$ closes a cycle when added to the forest $F_i$.
%Let $A_i=F_i\cup C_i$ denote the decomposition of $A_i$ into a spanning forest $F_i$ and some additional edges the create cycles $C_i$.
%Loops in $M/F$ are precisely the edges in $A\backslash F$. 
%As before,
%$k(M_i/F_i)\!-1\!=1-1=0= \!v(M_i/F_i)-\!k(M_i\backslash F_i^c)=\!v(M_i/F_i)-\!k(M_i\backslash A_i^c)$, the last equality by the fact that $F_i$ is a maximal spanning forest of $A_i$.
An edge $e\in A\backslash F$ is thus a loop of $M/F$  and when it is contracted its incident vertex $v$ splits into two new vertices $v_1$ and $v_2$. These vertices are either adjacent, in which case $v(M/(F\cup \{e\}))=v(M/F)+1$ while $k(M/F)=k(M/(F\cup\{e\}))$, or non-adjacent, in which case $v(M/(F\cup \{e\}))=v(M/F)+1$ and $k(M/(F\cup \{e\}))=k(M/F)+1$.  In both cases we have $k(M/(F\cup\{e\})\!-\!1\leq \!v(M/(F\cup\{e\}))-\!k(M\backslash F^c)$. 

It may be that, when contracting loop $e\in A\backslash F$ of $M/F$ on vertex $v$, a loop $e'\in A\backslash(F\cup\{e\})$ of $M/F$ becomes in $M/(F\cup\{e\})$ a non-loop edge joining $v_1$ and $v_2$. % ($e'$ is another loop incident with $v$). 
Then  $v(M/(F\cup \{e,e'\}))=v(M/(F\cup \{e\}))-1=v(M/F)$ while $k(M/(F\cup \{e,e'\}))=k(M/(F\cup \{e\}))=k(M/F)$. In this case $k(M/(F\cup\{e,e'\})\!-\!1\leq \!v(M/(F\cup\{e,e'\}))-\!k(M\backslash F^c)$, and again the desired inequality holds. 
 Any other edges $e''\in A\backslash (F\cup\{e,e'\})$ that are loops in $M/F$ are loops in $M/(F\cup\{e,e'\})$.
 
We may now repeat the argument: having contracted edges $B\subseteq A\backslash F$ to leave just loops, preserving the desired inequality $k(M/(F\cup B))\!-\!1\leq\!v(M/(F\cup B))-\!k(M\backslash F^c)$, choose an edge $e\in A\backslash (F\cup B)$ that is a loop of $M/(F\cup B)$. Eventually all the edges of $A$ are contracted and the inequality~\eqref{eq:components} holds.
%In this case the number of vertices increases, yet the number of connected components does not (as the newly created vertex is connected to the one from which it has been separated). However, when $e'$ % the first-loop-now-edge
%is contracted equality is again restored. By always contracting edges first before contracting the loops, one can see that the inequality holds.
%Since the inequality follows for each connected component, and all these parameters are additive over connected components, the inequality holds for maps.
%\end{proof}

%% file: proof_plane_101.tex
%!TEX root = /Users/crovellu/Dropbox/surface_tutte/First_surface_tutte/tutte_graphs_maps_s2_s3_ini_100.tex
\begin{proof}
For a plane map $M$ with underlying graph $\Gamma$ we have $g(M)=0$ and $g(M\backslash A^c)=0=g(M/A)$ for every $A\subseteq E$ (a consequence of Lemma~\ref{lem:genus_submaps}). Hence $\mathcal{T}(M;\mathbf x,\mathbf y)$ is a polynomial in $x,y,x_0,y_0$.

Using the defining subset expansion~\eqref{eq:surface_Tutte} for $\mathcal T(M;\mathbf x,\mathbf y)$ and equation~\eqref{eq:dual_nullity}, we have
\begin{equation}\label{eq:stutte_plane_1}
	\mathcal{T}(M;\mathbf x,\mathbf y)=\sum_{A\subseteq E}x^{r(M/A)}y^{n(M\backslash A^c)}x_0^{k(M/A)}y_0^{k(M\backslash A^c)}.
\end{equation}

We check that the exponents of $x,y,x_0$ and $y_0$ in equation~\eqref{eq:stutte_plane_1} are, comparing with the defining expansion of the Tutte polynomial~\eqref{eq:tutte_e1},  respectively equal to $k(\Gamma\backslash A^c)-k(\Gamma), n(\Gamma\backslash A^c), k(\Gamma\backslash A^c)$ and $n(\Gamma\backslash A^c)+k(\Gamma)$.

The exponent of $y$ is $n(M\backslash A^c)=n(\Gamma\backslash A^c)$ (using equation~\eqref{eq:mullity_del} above). %Deletion of an edge $e$ in $M$ gives a map $M\backslash e$ that is a plane embedding of $\Gamma\backslash e$. Hence  
%$$n(M\backslash A^c)=n(\Gamma\backslash A^c).$$
%he exponent of $y$ in \eqref{eq:stutte_plane_1}
%is clearly the same as the exponent of $(y-1)$ in
%\eqref{eq:tutte_e1}.

%Next, we claim that $r(M/A)=k(\Gamma)-k(\Gamma\backslash A^c)$. 
%
%First we show that 
%$r(M/A)=v(M/A)-k(M/A)=k(M)-k(M\backslash A^c)$, and then the claim follows by the fact that deletion in plane maps coincides with deletion in the underlying graph.
For the exponent of $x$, since $g(M)=0$ and $g(M/A)+g(M\backslash A^c)\leq g(M)$ we have $g(M/A)+g(M\backslash A^c)=g(M)$, whence $v(M/A)-k(M/A)=k(M\backslash A^c)-k(M)$ by Lemma~\ref{lem:genus_submaps}.
This implies
\begin{align*}r(M/A) &= k(M\backslash A^c)-k(M)\\
  &=k(\Gamma\backslash A^c)-k(\Gamma).\end{align*}
%since deletion in plane maps coincides with deletion in the underlying graph.

%We shall argue for each connected component of $A$. If $A$ is a forest $F$, the equality follows as $k(M/A)=k(M)$ and we have reduced all the different tree components of $F$ to just one vertex in $M/A$. Next, we observe that for any cycle in $A=F\cup E_0$ for $F$ a forest and each edge in $E_0$ closing a cycle on $F$, $v(M/A)=v(M/F)+|E_0|$ as all the edges in $E_0$ are loops in $M/F$, hence when contracted create an additional vertex. However, $k(M/A)=k(M/F)+|E_0|$ also have an additional component for each $e\in E_0$ with respect to $k(M/F)$, as when we contract the loop an additional component (and vertex) is created. The factor $k(M)-k(M\backslash A^c)$ is equal to $k(M)-k(M\backslash F^c)$ as adding the edges of $E_0$ does not change the number of connected components of $k(M\backslash F^c)$.

The exponent of $y_0$ in equation~\eqref{eq:stutte_plane_1} is equal to $k(M\setminus A^c)=k(\Gamma\backslash A^c)=[k(\Gamma\backslash A^c)-k(\Gamma)]+k(\Gamma)$. 

The dual of the identity $k(M\setminus A^c)-k(M)=r(M/A)$ is 
$k(M/A)-k(M)=n(M\setminus A^c)$ for each $A\subseteq E$,
from which we find that the exponent of $x_0$ is equal to $k(M/A)=n(M\backslash A^c)+k(M)=n(\Gamma\backslash A^c)+k(\Gamma)$.
% 
% 
% % ~(for Tutte and for Surface Tutte Definition~\ref{}).
% 
%  Since $r(M/A)=r(M)-r(M\backslash A^c)$ and $k(M/A)=k(M)-k(M\backslash A^c)$ [....]
\end{proof}

%% file: specializations_st_120.tex
\subsection{Flows and tensions}

We begin by giving analogues for maps of the specializations of the Tutte polynomial of a graph to the chomatic polynomial and the flow polynomial described at the end of Section~\ref{sec:Tutte_graph}. %To do this we need to give an informal definition of flows and tensions for maps (formal definitions are given in Section~\ref{sec:local_tensions_flows}, Definitions~\ref{def:local_tension} and~\ref{def:local_flow}).
 
Recall that a map $M$ determines a cyclic ordering of edges around each of its vertices and around each of its faces (and contractible cycles generally), given by following edges in anticlockwise order in the surface in which the underlying graph of $M$ is embedded. Let $M$ be given an arbitrary orientation of its edges. When following edges around a face of $M$ edges may either forward or backward, according as they are traversed in the same or opposite direction to their orientation. Likewise, at each vertex edges are either directed away from or towards the vertex.

\begin{definition}\label{def:flows_tensions}
Let $G$ be a finite group. 
A {\em (nowhere-identity) local $G$-tension} of $M$ is an assignment of (non-identity) values of $G$ to the edges of $M$ such that the elements of $G$ on the edges around a face taken in anticlockwise cyclic order, and in which values on backward edges are inverted, have product equal to the identity. (It does not matter at which edge one starts.) Likewise, a {\em (nowhere-identity) local $G$-flow} of $M$ is an assignment of (non-identity) values of $G$ to edges of $M$ such that for each vertex the product of values around it in anticlockwise order,  and in which values on incoming edges are inverted,  is equal to the~identity. 
\end{definition}
See Figure~\ref{fig:flowdef} for an illustration of Definition~\ref{def:flows_tensions}. 

\begin{figure}
\centering
\includegraphics[width=0.95\textwidth]{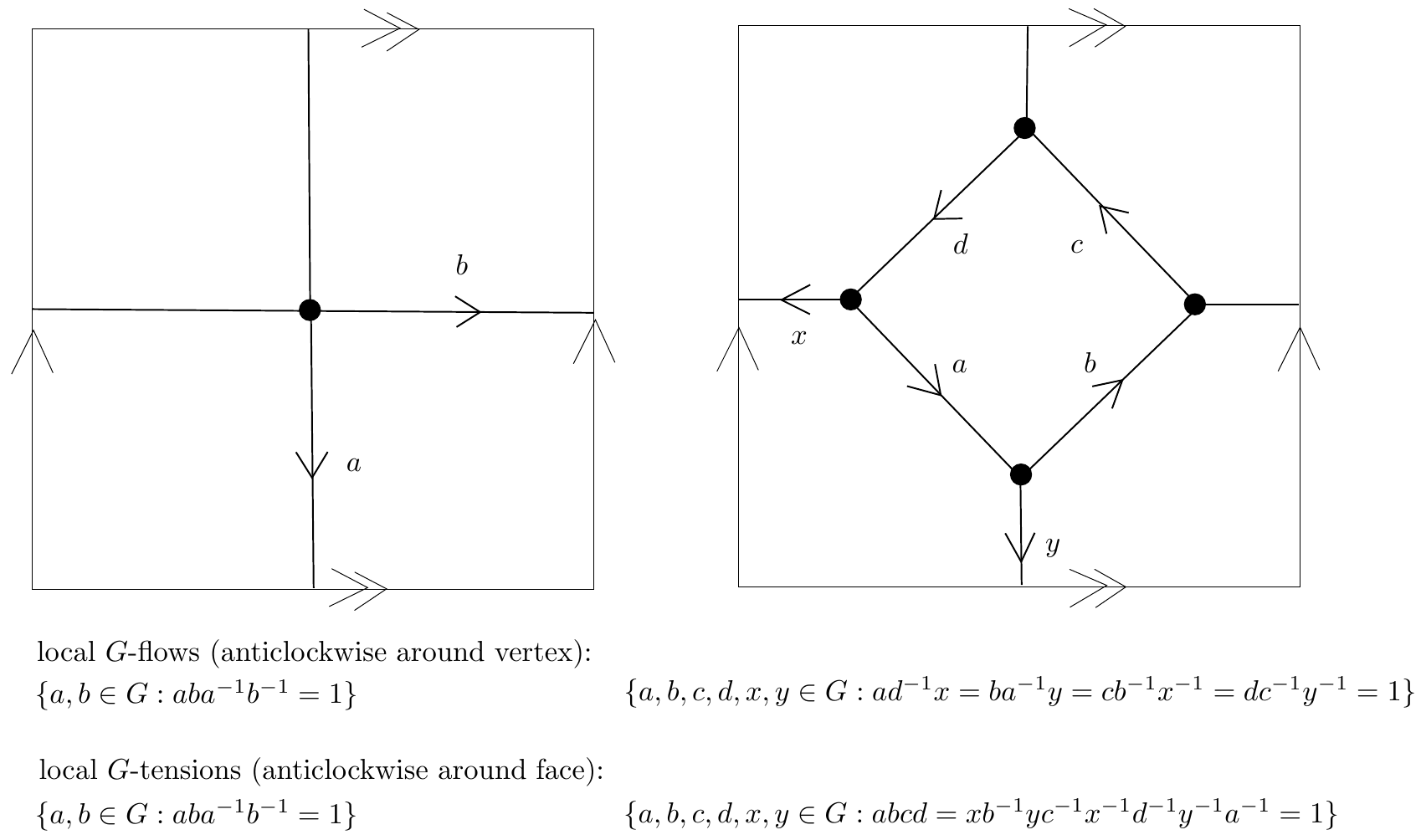}
\caption{Local $G$-flows and local $G$-tensions for two maps in the torus given a fixed arbitrary orientation of edges.}\label{fig:flowdef}
\end{figure}

When $M$ is a plane map and $G$ is an abelian group written additively (the identity is zero), nowhere-identity local $G$-tensions and nowhere-identity local $G$-flows coincide with nowhere-zero $G$-tensions and nowhere-zero $G$-flows of the underlying graph of $M$. For higher genus this correspondence no longer obtains. The qualifier ``local" refers in the case of tensions to the fact that we do not require the product of values around every cycle to equal the identity (just the facial walks) and in the case of flows that we do not require the product of values across every cutset to equal to identity (just the single vertex cutsets).

Definition~\ref{def:surface_tutte} of the surface Tutte polynomial gives a polynomial map invariant that has evaluations giving both the number of nowhere-identity local $G$-tensions and the number of nowhere-identity local $G$-flows. In this way we follow the example of Tutte in the way he defined the dichromate of a graph as a simultaneous generalization of the chromatic polynomial and flow polynomial. Recall~\cite{FH91} that to a finite group of order $n$ is associated a finite set of positive integers $\{n_i:i=1,\dots, \ell\}$, each $n_i$ a divisor of $n$ and such that $\sum n_i^2=n$, giving the dimensions of the irreducible representations of $G$ over $\mathbb C$. For an abelian group we have $n_i=1$ for each $1\leq i\leq \ell=n$.
In Section~\ref{ssec:n1flows} we prove the following:

\begin{theorem}\label{thm:spec_flows_tensions}
Let $M$ be a map, $\mathcal T(M;\mathbf x,\mathbf y)$ the surface Tutte polynomial of $M$, and $G$ a finite group the irreducible representations of which have dimensions $n_1,\dots,n_\ell$. % (thus $\sum n_i^2=|G|$). 
Then the number of nowhere-identity local $G$-tensions of $M$ is given by
$$(-1)^{e(M)-f(M)}\mathcal T(M;\mathbf x,\mathbf y),\quad \mbox{\rm with }\; x=-|G|, y=1,\quad x_g=-\frac{1}{|G|}\sum_{i=1}^\ell n_i^{2-2g},\:\: y_g=1,$$ and the number of nowhere-identity local $G$-flows by
$$(-1)^{e(M)-v(M)}\mathcal T(M;\mathbf x,\mathbf y),\quad \mbox{\rm with }\; x=1, y=-|G|,\quad x_g=1,\: y_g=-\frac{1}{|G|}\sum_{i=1}^\ell n_i^{2-2g},$$ 
for $g=0,1,2,\dots$.
\end{theorem}

\begin{remark}
As noted in Remark~\ref{rmk:finitely_many_var}, the polynomial $\mathcal T(M;\mathbf x,\mathbf y)$ is a polynomial in finitely many indeterminates: its specializations for counting flows and tensions involve only the variables $x_g, y_g$ for $g=0,1,\dots, g(M)$. In Theorem~\ref{thm:spec_flows_tensions} we can set $x_g, y_g=0$ for $g>g(M)$.
\end{remark}

\subsection{Quasi-trees of given genus}
%

%We now move on to analogues for maps of the evaluations $T(\Gamma;1,1)$, $T(\Gamma;2,1)$ and $T(\Gamma;1,2)$ of the Tutte polynomial of a graph $\Gamma$, which as is well known give respectively the number of  spanning trees, spanning forests and spanning connected subgraphs of $\Gamma$. 

Replacing $x_g$ by $x^{-2g}x_g$ and $y_g$ by $y^{-2g}y_g$ for $g=0,1,\dots$ in $\mathcal T(M;\mathbf x,\mathbf y)$, we obtain the following renormalization of the surface Tutte polynomial: 
\begin{definition} \label{def:surface_tutte_renorm}
Let $\mathbf x=(x, x_0,x_1,x_2,\dots)$, $\mathbf y=(y, y_0,y_1,\dots)$ be two infinite sequences of commuting indeterminates.

Given a  map $M=(V,E,F)$, define 
\begin{equation}\label{eq:surface_Tutte_renorm}\widetilde{\mathcal T}(M;\mathbf x,\mathbf y)=\sum_{A\subseteq E}x^{r(M/A)}y^{r^*(M\backslash A^c)}\prod_{\stackrel{\mbox{\rm \tiny conn. cpts $M_i$}}{\mbox{\rm \tiny of $M/A$}}}x_{g(M_i)}\prod_{\stackrel{\mbox{\rm \tiny conn. cpts $M_j$}}{\mbox{\rm \tiny of $M\backslash A^c$}}}y_{g(M_j)},\end{equation}
where $r(M)=v(M)-k(M)$, $r^*(M)=f(M)-k(M)$ and $A^c=E\setminus A$ for $A\subseteq E$.  
\end{definition}

\begin{proposition}\label{prop:genus_xy} Let $M$ be a connected map with $g(M)=g$. Let $h$ be an integer with $0\leq h\leq g$. 
Then the evaluation
of $\widetilde{\mathcal T}(M;\mathbf x,\mathbf y)$ at $x=y=0$, $x_{i}=0$ for $i\neq g-h$, $x_{g-h}=1$, $y_j=0$ for $j\neq h$, and $y_{h}=1$ is equal to the number of quasi-trees of $M$ of genus $h$ (which is also equal to the number of quasi-trees of $M^*$ of genus $g-h$.)
\end{proposition}
\begin{proof}
Let $A\subseteq E$ be such that it gives a nonzero contribution to the sum~\eqref{eq:surface_Tutte_renorm} with the given values assigned to the indeterminates $\mathbf x, \mathbf y$.
Then $r(M/A)+r^*(M\setminus A^c)=0$  (from the fact that $x=0=y$), each component of $M/A$ has genus $g-h$ and each component of $M\setminus A^c$ has genus $h$ (from the fact that $x_{g-h}=1=y_h$ while $x_i=0$ for $i\neq g-h$ and $y_j=0$ for $j\neq h$).
By additivity of the genus over connected components, this immediately implies that $g(M/A)\geq g-h$ and $g(M\setminus A^c)\geq h$.
Then by Lemma~\ref{lem:genus_submaps} we know that equality must hold, that is $g(M/A)= g-h$ and $g(M\setminus A^c)= h$.

Since rank and dual rank take non-negative values, we have $r(M/A)=0=r^*(M\backslash A^c)$, whence 
\begin{equation}\label{eq:vkfk} v(M/A)=k(M/A)\quad\mbox{ and }\quad f(M\backslash A^c)=k(M\backslash A^c).\end{equation}
As $g(M)=g(M/A)+g(M\setminus A^c)$ we know by Lemma~\ref{lem:genus_submaps},
\begin{equation}\label{eq:kf} k(M/A)+k(M\backslash A^c)= k(M)+f(M\backslash A^c),\end{equation}
and, dually, 
\begin{equation}\label{eq:vk} k(M/A)+k(M\backslash A^c)= k(M)+v(M/A).\end{equation}

From equations~\eqref{eq:vkfk} and~\eqref{eq:vk} we have $k(M\backslash A^c)=k(M)=1$ and from equations~\eqref{eq:vkfk} and~\eqref{eq:kf} we have $k(M/A)=k(M)=1$. 
Hence $M\backslash A^c$ is a quasi-tree and $h=g(M\backslash A^c)$, while $M^*\backslash A\cong (M/A)^*$ is a quasi-tree and $g-h=g(M/A)=g(M^*\backslash A)$.

Conversely, if $M\backslash A^c$ is a quasi-tree of genus $h$ (or $M^*\backslash A$ a quasi-tree of genus $g-h$) then $A$ contributes $1$ to the sum~\eqref{eq:surface_Tutte_renorm} with the given values assigned to the indeterminates.
Hence for $0\leq h\leq g$ the given evaluation is equal to $$\#\{A\subseteq E :  f(M\backslash A^c)=k(M\backslash A^c)=1, g(M\backslash A^c)=h\},$$
that is, the number of quasi-trees of $M$ of genus $h$.
\end{proof}

\begin{remark} The quasi-trees of maximum genus (genus zero) in a map $M$ form the bases of the matroid whose bases are the feasible sets of maximum (minimum) size in the $\Delta$-matroid associated with $M$ (see~\cite{B89}). Thus the evaluations of Proposition~\ref{prop:genus_xy} for $h\in\{0,g(M)\}$ are evaluations of the Tutte polynomial of the upper and lower matroids lying within the $\Delta$-matroid of~$M$.  
  \end{remark}
\begin{remark} As shown in~\cite{CKS11}, the ordinary generating function for quasi-trees of a connected map $M$ according to genus is given by evaluating the specialization 
$$q(M;t,Y)=\mathcal R(M;1,Y,tY^{-2})=\sum_{\stackrel{A\subseteq E}{k(M\backslash A^c)=1}}t^{g(M\backslash A^c)}Y^{n(M\backslash A^c)-2g(M\backslash A^c)}$$ of the Bollob\'as--Riordan polynomial %(which is a polynomial in $t$ and $Y$)
 at $Y=0$ (we have $n(M\backslash A^c)-2g(M\backslash A^c)=f(M\backslash A^c)-k(M\backslash A^c)$).
 The coefficients of $q(M;t,0)$ are the evaluations of $\widetilde{T}(M;\mathbf x,\mathbf y)$ given in Proposition~\ref{prop:genus_xy}. Let $\zeta$ be a primitive $(g(M)\!+\!1)$th root of unity. Then
the number of quasi-trees of genus $h$, evaluated in Proposition~\ref{prop:genus_xy}, is also given by
$$\frac{1}{g(M)\!+\!1}\sum_{j=0}^{g(M)}q(M;\zeta^j,0)\zeta^{-jh}.$$
\end{remark}

\subsection{Quasi-forests}\label{sec:specializations}
%In order to present the specializations of the surface Tutte polynomial 

%Note that $k(M\backslash A^c)=k(M)=1$ for a submap $M\backslash A^c$ of $M$ that is a quasi-tree.
A {\em quasi-forest} of $M$ is a submap of $M$ each of whose connected components is a quasi-tree. A {\em maximal quasi-forest} of a map $M$ is a quasi-forest each of whose components is a quasi-tree of a connected component of $M$. When $M$ is connected a maximal quasi-forest is a quasi-tree.  
When $M\backslash A^c$ is a maximal quasi-forest we have $k(M\backslash A^c)=k(M)$.

For the remaining specializations of the surface Tutte polynomial that follow it will be convenient to first specialize $\mathcal T(M;\mathbf x,\mathbf y)$ and its renormalization given in Definition~\ref{def:surface_tutte_renorm} to polynomials in four variables.
%by making the ``principal specialization" 
\begin{definition}\label{def:S}
We set $x_g=a^{g}$ and $y_g=b^{g}$ in $\mathcal T(M;\mathbf x,\mathbf y)$ (in which $\mathbf x=(x,x_0,x_1,\dots)$ and $\mathbf y=(y,y_0,y_1,\dots)$) to give the quadrivariate polynomial
\begin{equation}\label{eq:Q}\mathcal Q(M;x,y,a,b)=\sum_{A\subseteq E}x^{n^*(M/A)}y^{n(M\backslash A^c)}a^{g(M/A)}b^{g(M\backslash A^c)}.\end{equation}
Likewise, setting $x_g=a^{g}$ and $y_g=b^{g}$ in $\widetilde{\mathcal T}(M;\mathbf x,\mathbf y)$
\begin{equation}\label{eq:Q_renorm}\widetilde{\mathcal Q}(M;x,y,a,b)=\sum_{A\subseteq E}x^{r(M/A)}y^{r^*(M\backslash A^c)}a^{g(M/A)}b^{g(M\backslash A^c)}.\end{equation}
\end{definition}
The polynomials of Definition~\ref{def:S} %are the ``principal specializations" of $\mathcal{T}(M;\mathbf x, \mathbf y)$ and $\widetilde{\mathcal T}(M;\mathbf x,\mathbf y)$ respectively, They 
are simply related by
$$\mathcal{Q}(M;x,y,a,b)=\widetilde{\mathcal Q}(M;x,y,ax^2,by^2),$$
but it is useful to have notation for them both separately, since we shall be making evaluations where some of the variables $x,y,a,b$ are set to zero.

By Proposition~\ref{prop:surface_Tutte_plane}, if $M=(V,E,F)$ is a plane embedding of $\Gamma=(V,E)$ then
$\mathcal Q(M;x,y,a,b)=T(\Gamma;x+1,y+1)=\widetilde{\mathcal Q}(M;x,y,ax^2,by^2)$.

\begin{remark}\label{rmk:Delta_Q} $\Delta$-matroids are to maps as matroids are to graphs~\cite{B89, CMNR16}. 
%The definitions of $\mathcal Q(M;x,y,a,b)$ and $\widetilde{\mathcal Q}(M;x,y,a,b)$ just involve parameters of the $\Delta$-matroid underlying $M$ and so can be extended from maps (oriented ribbon graphs) to delta-matroids more generally. 
Just as the Las Vergnas polynomial, Bollob\'as--Riordan polynomial and Krushkal polynomial of a ribbon graph can be extended to $\Delta$-matroids, cf.~\cite[Section 6]{CMNR16}, this is also true for the polynomials $\mathcal Q(M;x,y,a,b)$ and $\widetilde{\mathcal Q}(M;x,y,a,b)$. 
A short explanation for $\widetilde{\mathcal Q}(M;x,y,a,b)$ is as follows.  
For a subset $A$ of the edges of $M$ the coefficient of $x$ is given by $r(M/A)$, which in $\Delta$-matroid terminology (see~\cite{CMNR16}) is nothing other than the rank of the lower matroid of the $\Delta$-matroid underlying $M/A$. 
The coefficient of $a$ is the genus of $M/A$, which in $\Delta$-matroid terminology is equal to half the \emph{width} of the $\Delta$-matroid underlying~$M/A$. The coefficients of $y$ and $b$ are similarly expressed in terms of parameters of the $\Delta$-matroid underlying~$M^*/A^c$. 
%We do not know whether it is possible to express these coefficients in terms of parameters of the $\Delta$-matroid underlying $M$.

The surface Tutte polynomial $\mathcal T(M;\mathbf x,\mathbf y)$ cannot be extended to $\Delta$-matroids because its definition involves the genera of the connected components of $M$ and cannot be made independent of these; this is similar to how the $U$-polynomial~\cite{NW99} multivariate generalization of the Tutte polynomial of a graph $\Gamma$ does not lift to matroids more generally as its definition ineluctably involves the ranks of the connected components of $\Gamma$.
\end{remark}

%The ``principal specialization" of the surface Tutte polynomial $\mathcal{Q}(M;x,y,a,b)$ is defined in Definition~\ref{def:S}. 
The specializations and evaluations of $\mathcal{Q}(M;x,y,a,b)$ and $\widetilde{Q}(M;x,y,a,b)$ (and hence of $\mathcal{T}(M;\mathbf x,\mathbf y)$) that follow are related to the Tutte polynomial specializations
$$T(\Gamma;x+1,1)=\sum_{\stackrel{A\subseteq E}{r(A)=|A|}}x^{r(\Gamma)-|A|},$$
$$T(\Gamma;1,y+1)=\sum_{\stackrel{A\subseteq E}{r(A)=r(E)}}y^{|A|-r(\Gamma)},$$
giving respectively generating functions for spanning forests of $\Gamma$ according to their number of edges and for connected spanning subgraphs.
We need two more definitions though. 
A \emph{bridge} of a map $M=(V,E,F)$ is a an edge $e$ such that $M\!\setminus \!e$ has more components than $M$.
By definition a bridge is incident to only one face.
A \emph{dual bridge} of a map $M=(V,E,F)$ is a loop $e$ of $M$ such that $e$ is a bridge of the dual of $M$.
%We note that contracting a bridge $g$ of a map $M$ does not affects its genus. 

\begin{proposition}\label{prop:plane_quasi-trees_bouquets}
For a map $M$, 
$$\mathcal Q(M;x,0,1,1)=x^{2g(M)}\sum_{\stackrel{A\subseteq E: \:\mbox{\rm \tiny conn. cpts of}}{\mbox{\rm \tiny $M\backslash A^c$ plane quasi-trees}}}x^{r(M)-|A|}.$$
$$\mathcal Q(M;0, y,1,1)=y^{2g(M)}\sum_{\stackrel{A\subseteq E:\: \mbox{\rm \tiny conn. cpts of}}{\mbox{\rm \tiny $M/A$ plane bouquets}}}y^{r^*(M)-|A|}.$$
(A plane quasi-tree consists solely of bridges, and so corresponds to an embedding of a tree. A plane bouquet consists solely of dual bridges.)
\end{proposition}
\begin{proof} 
%When $g(M)>0$, 
From equation~\eqref{eq:Q} 
$$\mathcal Q(M;x,0,1,1)=\sum_{\stackrel{A\subseteq E:}{\mbox{\rm \tiny $n(M\backslash A^c)=0$}}}x^{n^*(M/A)}$$ 
$$\mathcal Q(M;0,y,1,1)=\sum_{\stackrel{A\subseteq E:}{\mbox{\rm \tiny $n^*(M/ A)=0$}}}y^{n(M\backslash A^c)}.$$ 

%The equality $e(M\backslash A^c)-v(M\backslash A^c)+k(M\backslash A^c)=0$ holds if and only if the submap $M\backslash A^c$ contains no cycles: 

We prove the first identity in the proposition statement; the second follows by duality, with $\mathcal Q(M;0,y,1,1)=\mathcal Q(M^*;y,0,1,1)$.

By Euler's relation, $n(M\backslash A^c)=e(M\backslash A^c)-v(M\backslash A^c)+k(M\backslash A^c)=f(M\backslash A^c)-k(M\backslash A^c)+2g(M\backslash A^c)$ and $f(M\backslash A^c)\geq k(M\backslash A^c)$ with equality if and only if each component of $M\backslash A^c$ has just one face. Thus for $n(M\backslash A^c)=0$ to hold $M\backslash A^c$ must be a disjoint union of plane quasi-trees (i.e., trees). %, all the one-faced edges being bridges). % (all edges are one-faced). 
Given that the edges of $A$ form a disjoint union of plane quasi-trees, we have, beginning with Euler's relation,
\begin{align*}n^*(M/A)&=e(M/A)-f(M/A)+k(M/A)\\
 & = v(M/A)-k(M/A)+2g(M/A)\\
 & = v(M)-|A|-k(M)+2g(M)\\
&=2g(M)+r(M)-|A|,\end{align*} 
since contracting a bridge of a map changes neither connectivity nor genus. This establishes the result.
%We have $v(M/A)\geq k(M/A)$ with equality if and only if each connected component of $M/A$ has just one vertex. Thus for $n^*(M/A)=0$ to hold $M/A$ must be a disjoint union of plane bouquets. % (all the edges in $A$ are dual bridges. 
%Then
%\begin{align*}n(M\backslash A^c) & =e(M\backslash A^c)-v(M\backslash A^c)+k(M\backslash A^c) \\
%& = f(M\backslash A^c)-k(M\backslash A^c)+2g(M/A)\\
% & = f(M)-|A|-k(M)+2g(M)\\
%& =2g(M)+r^*(M)-|A|\\
%& = n(M)-|A|\end{align*} 
%since deleting a dual bridge of a map changes neither connectivity nor genus, and finally using the fact that
%$r^*(M)=f(M)-k(M)=n(M)-2g(M)$. %(The nullity of a map is defined by $n(M)=e(M)-v(M)+k(M)$.) 
%Thus every edge of $A$ is a bridge. Since contracting a bridge of a map changes neither connectivity nor genus, the fact that $M/A$ is a disjoint union of $k(M/A)=k(M)$ plane bouquets implies that $M$ must have each connected component a plane map. 
\end{proof}

\begin{corollary}\label{cor:plane_00}
For a map $M=(V,E,F)$ with underlying graph $\Gamma=(V,E)$, the constant term of the specialization $\mathcal Q(M;x,y,1,1)$ of $\mathcal{T}(M;\mathbf x,\mathbf y)$ is given by
$$\mathcal Q(M;0,0,1,1)=\begin{cases} T(\Gamma;1,1) &  \mbox{if $M$ is a plane embedding of planar graph $\Gamma$,}\\
 0 & \mbox{otherwise.}\end{cases}$$ 
Further, we have the following evaluations for any map $M=(V,E,F)$:
$$\mathcal Q(M;1,0,1,1)=\#\{A\subseteq E:\mbox{\rm connected components of } M\backslash A^c \:\mbox{\rm are plane quasi-trees}\},$$ 
$$\mathcal Q(M;0,1,1,1)=\#\{A\subseteq E:\mbox{\rm connected components of } M/A \:\mbox{\rm are plane bouquets}\}.$$ 
\end{corollary}
\begin{remark} \label{rmk:Tutte_anal} The last two evaluations are analogous (and for plane maps, identical) to the following Tutte polynomial evaluations for a graph $\Gamma=(V,E)$ giving for connected $\Gamma$ the number of spanning forests and number of connected spanning subgraphs: 
\begin{align*}T(\Gamma;2,1) & = \#\{A\subseteq E: n(\Gamma\backslash A^c)=0\}\\
& =\#\{A\subseteq E:\mbox{\rm connected components of } \Gamma\backslash A^c \:\mbox{\rm are trees}\},\end{align*}
and 
\begin{align*}T(\Gamma;1,2)& =  \#\{A\subseteq E: r(\Gamma\backslash A^c)=r(\Gamma)\}. \\
& =\#\{A\subseteq E:\mbox{\rm connected components of } \Gamma/A \:\mbox{\rm are single vertices with loops}\}.
\end{align*} 
% (For the latter, if $A$ is a connected spanning subgraph of connected $\Gamma$ then $A^c$ contains no bridges and contracting all the edges in $A$ must reduce the graph to a single vertex.)
%\footnote{There was the text, which is now deleted as it was rather confusing:
%\begin{align*}
%& =\#\{A\subseteq E:\mbox{\rm connected components of } \Gamma/A \:\mbox{\rm are single vertices with loops}\}.
%\end{align*} 
%(For the latter, if $A$ is a connected spanning subgraph of connected $\Gamma$ then $A^c$ contains no bridges and contracting all the edges in $A$ must reduce the graph to a single vertex.)
%}
\end{remark}
\begin{proof}[Proof of Corollary~\ref{cor:plane_00}]
This is immediate from Proposition~\ref{prop:plane_quasi-trees_bouquets} with $x=1$. 
%when $M$ has a non-plane component, there can be no contribution by $A\subseteq E$ to the constant term of $\mathcal S(M;x,y)$.. 
By Proposition~\ref{prop:surface_Tutte_plane}, when $M$ is a plane embedding of planar $\Gamma$, $\mathcal{Q}(M;0,0,1,1)=T(\Gamma;1,1)$. % counts the number of spanning trees of the planar graph $\Gamma$ underlying $M$.
\end{proof}

\begin{proposition}\label{prop:quasi-trees_bouquets}
For a map $M$,
$$\widetilde{\mathcal Q}(M;x,0,1,1)=\sum_{\stackrel{A\subseteq E: \:\mbox{\rm \tiny conn. cpts of}}{\mbox{\rm \tiny $M\backslash A^c$ quasi-trees}}}x^{r(M/A)},$$
$$\widetilde{\mathcal Q}(M;0, y,1,1)=\sum_{\stackrel{A\subseteq E:\: \mbox{\rm \tiny conn. cpts of}}{\mbox{\rm \tiny $M/A$ bouquets}}}y^{r^*(M\backslash A^c)}.$$%\footnote{Lluis: these evaluations should be at x,0,1,1 and 0,y,1,1}
\end{proposition}
\begin{proof} %By Lemma~\ref{lem:rank_geom_dual} and the fact that $(M\backslash A^c)^*\cong M^*/A^c$ (using Proposition~\ref{prop:dual_del_con}) 
%we have $$n((M/A)^*)-2g(M/A)=r(M/A)$$
%and
%$$n(M\backslash A^c)-2g(M\backslash A^c)=r((M\backslash A^c)^*)=r(M^*/A^c).$$
%Inspecting the defining equation~\eqref{eq:Q} for $\mathcal{Q}(M;x,y,a,b)$ the result follows.
This follows from the defining equation~\eqref{eq:Q_renorm} for $\widetilde{\mathcal Q}(M;x,y,a,b)$ given in Definition~\ref{def:S} and the fact that $r^*(M\backslash A^c)=f(M\backslash A^c)-k(M\backslash A^c)=0$ if and only if each connected component of $M\backslash A^c$ has just one face, i.e. they are quasi-trees, and dually $r(M/A)=v(M/A)-k(M/A)=0$ if and only if the connected components of $M/A$ each have one vertex, i.e. they are bouquets. 
\end{proof}

\begin{corollary}\label{cor:Q_plane_1010}
For a map $M=(V,E,F)$,  
$$\widetilde{\mathcal Q}(M;1,0,1,1)=\#\{A\subseteq E:\mbox M\backslash A^c \:\mbox{\rm is a quasi-forest}\},$$ 
$$\widetilde{\mathcal Q}(M;0,1,1,1)=\#\{A\subseteq E: \mbox{\rm connected components of }M/A \:\mbox{\rm are bouquets}\}.$$
\end{corollary}
%The Tutte polynomial of a graph $\Gamma=(V,E)$ is given by
%$$T(\Gamma;x,y)=\sum_{A\subseteq E}(x-1)^{r(\Gamma)-r(\Gamma\backslash A^c)}(y-1)^{r(\Gamma^*)-r(\Gamma^*\backslash A)},$$
%where $\Gamma^*$ is the matroid dual of $\Gamma$ (only isomorphic to the surface dual for planar graph $\Gamma$ embedded in the plane). The Krushkal polynomial is defined by [see Chun et al. arXiv paper]
%$$\sum_{A\subseteq E}(x-1)^{r(\Gamma)-r(\Gamma\backslash A^c)}y^{r(\Gamma^*)-r(\Gamma^*\backslash A)}a^{g(\Gamma\backslash A^c)}b^{g(\Gamma^*\backslash A)},$$
%where $\Gamma^*$ denotes matroid dual in first appearance and surface dual in second [?].

The evaluations of Corollary~\ref{cor:Q_plane_1010} are analogous (and for plane maps identical) with the evaluations $T(\Gamma;2,1)$ and $T(\Gamma;1,2)$ of the Tutte polynomial as the number of spanning forests and connected spanning subgraphs of $\Gamma$. (Compare Remark~\ref{rmk:Tutte_anal} concerning the evaluations of Corollary~\ref{cor:plane_00}.)
%By Proposition~\ref{prop:Q_del_con}, the number of spanning quasi-trees $q(M)=\mathcal{Q}(M;1,0,1,0)$ of a map $M$ satisfies the recurrence
%$$q(M)=
%\begin{cases} q(M\backslash e)+q(M/e) & \mbox{\rm if $e$ is an ordinary edge or interlacing two-faced loop,}\\
%q(M\backslash e) & \mbox{\rm if $e$ is an interlacing one-faced non-loop, dual bridge}\\
%& \hspace{2cm}\mbox{\rm  or one-faced loop,}\\
%2q(M/e) & \mbox{\rm if $e$ is a bridge.}
%\end{cases}
%$$
%
%

%Having now singled out our main results on the surface Tutte polynomial in this paper, we move on to their proofs. For this we require some preparation, beginning with how we define maps.

%The surface Tutte polynomial $\mathcal{T}(M;\mathbf x, \mathbf y)$ specializes to the Bollolb\'as--Riordan polynomial and Krushkal polynomial for plane maps (Proposition~\ref{prop:plane_Tutte} and Proposition~\ref{prop:plane_Kruskal}).
%For non-plane maps the surface Tutte polynomial does not specialize to the Tutte polynomial of the underlying graph. [PROOF...]\footnote{Lluis: this is proven in the Independence from Tutte section}

%\input{indep_tutte_101.tex}

%% file: proofs_100.tex
\section{Enumerating flows and tensions}\label{sec:proofs}

In this section $G$ will be a finite group with identity~$1$. 
%When $G$ is the symmetric group of degree $n$ conjugacy classes are in correspondence with partitions of $n$ (cycle lengths in a cycle decomposition of a permutation). Inequivelent irreducible representations. Young tableaux.  

\begin{definition}
For a map $M=(V,E,F)$, let $p_G^1(M)$ denote the number of local $G$-tensions of $G$ (allowing the identity $1$) and $p_G(M)$ the number of nowhere-identity local $G$-tensions of $M$ (see Definition~\ref{def:flows_tensions}). %\footnote{Lluis: added (see Definition~\ref{def:flows_tensions}).}
Dually, let $q_G^1(M)$ denote the number of local $G$-flows of $M$ and $q_G(M)$ the number of nowhere-identity local $G$-flows of $M$.
\end{definition}
By definition of the dual map %and Proposition~\ref{prop:comb_map_dual},
we have
\begin{equation*}%\label{eq.q_dual_p}
	q_G^1(M)=p_G^1(M^*)
\end{equation*}
and %hence also
\[q_G(M)=p_G(M^*).\]

The main goal of this section is to establish formulas for $q_G(M)$ and $p_G(M)$ in terms of the size of the group $G$ and the dimensions of its irreducible representations.

By partitioning local $G$-flows according to the set of edges $A$ on which the flow value is equal to the identity,
$$q_G^1(M)=\sum_{A\subseteq E}q_G(M\backslash A).$$
Then by the inclusion-exclusion principle
\begin{equation}
q_G(M)=\sum_{A\subseteq E}(-1)^{|A^c|}q_G^1(M\backslash A^c).\label{eq:in-ex}
\end{equation}
So we see that it suffices to find formulas for $q_G^1(M)$, which is the main focus of the following subsection.
In Subsection~\ref{ssec:n1flows} we give formulas for $q_G(M)$ and $p_G(M)$ and discuss some special cases.

\subsection{Enumerating local $G$-flows}\label{ssec:flows}
Here we establish the following result:
\begin{theorem}\label{thm:flow count}
Let $G$ be a finite group with irreducible representation of dimensions $n_\ell$.
Let $M=(V,E,F)$ be a connected map.
Then the number of local $G$-flows of $M$ is given by
\begin{equation}\label{eq:flow count}
q^1_G(M)=|G|^{|E|-|V|}\sum_{\ell} n_\ell^{\chi(M)}.
\end{equation}
\end{theorem}
\begin{remark}
This result has already implicitly appeared in \cite{MY05}, but we will give a proof below so as to make 
our paper as self contained as possible. 
\end{remark}
Before giving a proof of Theorem~\ref{thm:flow count}, let us further remark that it includes a remarkable result due to Mednyh \cite{M78} (Theorem~\ref{thm:Mednyh} here), which we now explain.
Let $M_{g}$ be the map given by a single vertex with $2g$ loops $e_1,\ldots,e_{2g}$ attached to it such that the vertex rotation is given by $(e_1,e_2,e_1,e_2,\ldots,e_{2g-1},e_{2g},e_{2g-1},e_{2g}$). (Note that the genus of $M_g$ is clearly equal to $g$.)
A local $G$-flow of $M_{g}$ is a solution in $G$ to the equation
\begin{equation}\label{eq:Gflow_Dg0} 
[a_1,b_1]\cdots [a_g,b_g]=1,
\end{equation}
where $[a,b]=aba^{-1}b^{-1}$ is the commutator of $a$ and $b$ in $G$.

If $\Sigma$ is an orientable compact surface of genus $g$, its fundamental group $\pi_1(\Sigma)$ has the presentation  
$$\pi_1 (\Sigma)\cong\langle a_1, b_1, \dots, a_g, b_g \;:\: [a_1,b_1]\cdots [a_g,b_g]=1\rangle.$$
(See for example~\cite[p.51]{H02}.) This implies that solutions in $G$ to equation~\eqref{eq:Gflow_Dg0} are exactly the  homomorphisms from $\pi_1(\Sigma)$ to $G$.
Let us denote the set all homomorphisms from $\pi_1(\Sigma)$ to $G$ by ${\rm Hom}(\pi_1(\Sigma,G)$.
Thus we have $q^1_G(M_g)=|{\rm Hom}(\pi_1(\Sigma),G)|$, and so Theorem~\ref{thm:flow count} implies the following result of 
Frobenius~\cite{F1896} (for $g=1$) and Mednyh~\cite{M78} (for $g>1$) (see e.g. the survey~\cite[Sect. 7]{J98} for more details):
%
%\footnote{Previously here: (same but with $\Sigma_g$ instead of $\Sigma$).
%
%If $\Sigma_g$ is an orientable compact surface of genus $g$, its fundamental group $\pi_1(\Sigma_g)$ has the presentation  
%$$\pi_1 (\Sigma_g)\cong\langle a_1, b_1, \dots, a_g, b_g \;:\: [a_1,b_1]\cdots [a_g,b_g]=1\rangle.$$
%(See for example~\cite[p.51]{H02}.) This implies that solutions in $G$ to equation~\eqref{eq:Gflow_Dg0} are exactly homomorphisms from $\pi_1(\Sigma_g)$ to $G$.
%Thus we have $q^1_G(M_g)=|{\rm Hom}(\pi_1(\Sigma_g),G)|$,\footnote{Lluis: added the $^1$ at the $q$.} and so Theorem~\ref{thm:flow count} implies the following result of 
%Frobenius~\cite{F1896} (for $g=1$) and Mednyh~\cite{M78} (for $g>1$) (see e.g. the survey~\cite[Sect. 7]{J98} for more details):} % in terms of the dimensions of irreducible representations of $G$: 
\begin{theorem}\label{thm:Mednyh} Let $\Sigma$ be a surface of genus $g>0$ and $G$ a finite group with dimensions of irreducible representations $n_\ell$. Then
  $$\frac{|{\rm Hom}(\pi_1(\Sigma),G)|}{|G|}=\sum_{\ell}\left(\frac{|G|}{n_\ell}\right)^{2g-2}.$$
\end{theorem}
\begin{remark} The numbers $\frac{|G|}{n_\ell}$ in Theorem~\ref{thm:Mednyh} are positive integers.

When $g=0$ the formula still holds: when $\Sigma$ is the sphere, $|{\rm Hom}(\pi_1(\Sigma),G)|=1=\frac{1}{|G|}\sum_\ell n_\ell^2$, which is a well known result from representation theory~\cite[p. 18, Corollary~2]{Ser}.
\end{remark} 
\begin{remark}
Although Mednyh's theorem is a consequence of our Theorem~\ref{thm:flow count}, it also possible to derive  Theorem~\ref{thm:flow count} from Mednyh's theorem, cf. \cite{MY05}.
\end{remark}

We now turn to the proof of Theorem \ref{thm:flow count}.
As mentioned above we prove it using representation theory. 
We refer the reader to the book by Serre~\cite{Ser} for definitions and background on representation theory. 
(In fact, the first twenty pages of~\cite{Ser} contain everything that we need.)
We shall use the following three facts.
The regular character $\chi_\reg$ of a finite group $G$ satisfies (\cite[p.18, Proposition~2]{Ser})
{\begin{equation}\label{eq:one to reg}
\chi_\reg(g)=\left\{\begin{array}{cl} |G| &\text{ if } g=1,\\ 0&\text{ otherwise}\end{array}\right.
\end{equation}
for all $g\in G$.
Let $C(G)$ denote the set of all irreducible representations (up to isomorphism) of $G$. 
Then (\cite[p.18, Corollary 1]{Ser}
\begin{equation}
\chi_\reg=\sum_{\rho\in C(G)}\dim(\rho)\chi_\rho.	\label{eq:character decomposition}
\end{equation}
Let $\rho,\rho'\in C(G)$, which we assume to be in matrix form.
Then for $i,j\in [\dim(\rho)]$ and $i',j'\in [\dim(\rho')]$ we have (\cite[p.14, Corollary 3]{Ser})
\begin{equation}
\sum_{g\in G} \rho(g)_{i,j}\rho'(g^{-1})_{j',i'}=\left\{\begin{array}{cl} \frac{|G|}{\dim(\rho)}& \text{ if } \rho=\rho',i=i'\text{ and } j=j',\\0&\text{otherwise.} \end{array}\right. \label{eq:orthogonality}
\end{equation}

For a map $M=(V,E,F)$ we denote the vertex rotation at $v\in V$ (of the multiset of edges containing $v$ $\delta(v)$) %\footnote{Lluis: added `` of the edges containing $v$ $\delta(v)$''.} 
by $\pi_v$. So for an edge $e\in \delta(v)$ $\pi_v(e)$ is the next edge in $\delta(v)$ with respect to the anticlockwise cyclic order with respect to the orientation of the surface.
We say that for $v\in V$ and $e\in \delta(v)$ the pair $(v,e)$ \emph{belongs} to a face $f$ of $M$ if $f$ is to the left %\footnote{Lluis: shouldn't it be ``to the left of e''?} 
of $e$ when traversing $e$ towards $v$.

Now we are ready to prove the result. 
\begin{proof}[Proof of Theorem~\ref{thm:flow count}]
Recall that we assume that the edges of $M$ have been given an arbitrary direction.
Let us define for $v\in V$ $\eps_v:\delta(v)\to \{-1,1\}$ by 
\[
\eps_v(e)=\left \{\begin{array}{rl}-1& \text{ if } e \text{ is an incoming arc, }\\1&\text{ if } e \text{ is an outgoing arc.}\end{array}\right.
\]
Let us denote the vertices and edges of the underlying graph of $M$ by $V$ and $E$ respectively. Then we can write 
\[
q^1_G(M)=\sum_{\phi:E\to G}\prod_{v\in V} {\bf1}\bigg(\prod_{e\in \delta(v)}\phi(e)^{\eps_v(e)}\bigg),
\]
where ${\bf1}(g)$ is equal to $1$ if and only if $g=1$ and zero otherwise.
By \eqref{eq:one to reg} we can rewrite this as follows
\begin{align*}%\label{eq:step 1}
q^1_G(M)=\sum_{\phi:E\to G}|G|^{-|V|}\prod_{v\in V} \chi_\reg\bigg(\prod_{e\in \delta(v)}\phi(e)^{\eps_v(e)}\bigg).
\end{align*}
Now using \eqref{eq:character decomposition} we obtain that 
\begin{align*}
|G|^{|V|}q^1_G(M)=&\sum_{\phi:E\to G}\prod_{v\in V} \sum_{\rho\in C(G)}\dim(\rho)\chi_\rho\bigg(\prod_{e\in \delta(v)}\phi(e)^{\eps_v(e)}\bigg)
\\
=&\sum_{\phi:E\to G}\sum_{\kappa:V\to C(G)}\prod_{v\in V}\dim(\kappa(v))\chi_{\kappa(v)}\bigg(\prod_{e\in \delta(v)}\phi(e)^{\eps_v(e)}\bigg).
\end{align*}
Let us now fix an assignment $\kappa:V\to C(G)$ and look at its contribution to the sum above.
This contribution is given by (using that representations are multiplicative and where $\tr$ denotes the trace)
\begin{align}
&\sum_{\phi:E\to G}\prod_{v\in V}\dim(\kappa(v))\tr\bigg(\prod_{e\in \delta(v)}\kappa(v)\big(\phi(e)^{\eps_v(e)}\big)\bigg)=\label{eq:contribution}
\\
&\sum_{\phi:E\to G}\prod_{v\in V}\dim(\kappa(v))\cdot \sum_{\psi:\delta(v)\to [\dim(\kappa(v))]} \bigg(\kappa(v)\big(\phi(e)^{\eps_v(e)}\big)\bigg)_{\psi(e),\psi(\pi_v(e))}=\nonumber
\\
&\sum_{\phi:E\to G}\sum_{\substack{v\in V\\\psi_v:\delta(v)\to\dim[\kappa(v)]}} \prod_{v\in V}\dim(\kappa(v))\cdot \prod_{e\in \delta(v)}\bigg(\kappa(v)\big(\phi(e)^{\eps_v(e)}\big)\bigg)_{\psi_v(e),\psi_v(\pi_v(e))}=\nonumber
\\
&\sum_{\substack{v\in V\\\psi_v:\delta(v)\to\dim[\kappa(v)]}}\prod_{v\in V}\dim(\kappa(v))\cdot \prod_{e=(u,v)\in E}\cdot \sum_{g\in G}(\kappa((u)(g))_{\psi_u(e),\psi_u(\pi_u(e))}(\kappa(v)(g^{-1}))_{\psi_v(e),\psi_v(\pi_v(e)))}.\nonumber
\end{align}
By \eqref{eq:orthogonality}, the last sum on the last line of \eqref{eq:contribution} is zero unless $\kappa(v)=\kappa\in C(G)$ for all $v\in V$ and additionally $\psi_u(e)=\psi_v(\pi_v(e)))$ and $\psi_v(e)=\psi_u(\pi_u(e)))$.
So to get a nonzero contribution we need that $\psi_u(e)$ only depends on the face of $M$ to which $(u,e)$ it belongs and that $\kappa$ is constant on $V$. 
Each such assignment $\psi$ then gives a contribution of $\dim(\kappa)^{|V|-|E|}|G|^{|E|}$ and so in total for 
constant $\kappa$ we get that \eqref{eq:contribution} is equal to $\dim(\kappa)^{|V|+|F|-|E|}|G|^{|E|}$.
Summing this over all irreducible representations of $G$ and dividing by $|G|^{|V|}$, we obtain the desired expression for $q^1_G(M)$.
\end{proof}

\subsection{Enumerating nowhere identity $G$-flows and tensions}\label{ssec:n1flows}
 In the previous subsection we enumerated local $G$-flows. 
We shall now use this to prove the following result, which immediately implies Theorem~\ref{thm:spec_flows_tensions}. %After this we discuss some of its special cases.
\begin{theorem}\label{thm:n1flows}
Let $G$ be a group with irreducible representations of dimensions $n_\ell$.
\begin{enumerate}
\item[(i)] The number of nowhere-identity local $G$-flows of a map $M=(V,E,F)$ is given by 
\begin{equation}\label{eq:n1flows} 
q_G(M)=\sum_{A\subseteq E}(-1)^{|A^c|}|G|^{n(M\backslash A^c)}\prod_{\mbox{\rm \tiny conn. cpts $M_j$ of $M\backslash A^c$}}\left(\frac{1}{|G|}\sum_\ell n_\ell^{\chi(M_j)}\right),
\end{equation}
where $A^c=E\setminus A$ and $\chi(M_j)=2-2g(M_j)$. 
\item[(ii)] The number of nowhere-identity local $G$-tensions of a map $M=(V,E,F)$ is given by 
$$p_G(M)=\sum_{A\subseteq E}(-1)^{|A|}|G|^{n^*(M/A))}\prod_{\mbox{\rm \tiny conn. cpts $M_i$ of $(M/A)$}}\left(\frac{1}{|G|}\sum_\ell n_\ell^{\chi(M_i)}\right).$$
\end{enumerate}
\end{theorem}

\begin{remark}
The subset $A$ is complemented in the summation~\eqref{eq:n1flows} for later convenience. The submap $M\backslash A^c$ is the restriction of $M$ to edges in $A$. 
\end{remark}
\begin{proof}
The statement for $q_G(M)$ follows directly from Theorem~\ref{thm:flow count} by \eqref{eq:in-ex}.
To prove the statement for $p_G(M)$, note that $e(M^*\backslash A^c)=|A|=e(M/A^c)$, $k(M^*\backslash A^c)=k((M^*\backslash A^c)^*)=k(M/A^c)$, and $v(M^*\backslash A^c)=f((M^*\backslash A^c)^*)=f(M/A^c)$, and so we see that $n(M^*\backslash A^c)=n^*(M/A^c)$.
Then, as the components of $M^* \setminus A^c$ are the components of $(M^*\setminus A^c)^*=M/A^c$, using that $p_G(M)=q_G(M^*)$, and replacing $A^c$ by $A$, we arrive at the desired equality for $p_G(M)$.
\end{proof}

We now discuss some special cases of Theorem~\ref{thm:n1flows}.
\begin{example}\label{prop:plane_local_flows} 
If $M=(V,E,F)$ is a plane map, $\chi(M_j)=2$ for all choices of $A$ and $j$ in equation~\eqref{eq:n1flows}.
Combining this with the fact that $\sum n_\ell^2=|G|$, $q_G(M)$ coincides for arbitrary finite group~$G$ with the flow polynomial of the planar graph $\Gamma=(V,E)$ evaluated at $|G|$. 

Similarly, $p_G(M)$ also depends only on $|G|$.
This can also be seen from that for a plane map $M$ local tensions are global, and we have the correspondence between $G$-tensions and vertex $G$-colourings of the graph underlying $M$: this make evident that the number of nowhere-identity $G$-tensions of a plane map depends only on $|G|$ and not on the structure of $G$ (proper vertex $G$-colourings depend only on distinctness of colours, that is, on $G$ as a set). 
\end{example}
\begin{example}\label{ex:abelian_local_flows}
When $G$ is abelian, $n_\ell=1$ for $i=1,\dots, |G|$, and as $\sum_{i=1}^{|G|} 1^{\chi(M_j)}=|G|$ we obtain for any map $M$ the same expression as for plane maps with arbitrary finite group~$G$,
$$q_G(M)=\sum_{A\subseteq E}(-1)^{|A^c|}|G|^{n(M\backslash A^c)} =\phi(\Gamma;|G|),$$ where $\phi$ denotes the flow polynomial.
Thus the number of local abelian $G$-flows of a map $M$ is equal to the number of global $G$-flows of $M$, evaluated by the flow polynomial of the graph underlying $M$. %(Edge deletion in $M$ corresponds to edge deletion in the underlying graph of $M$.) 

%As for flows (Example~\ref{ex:abelian_local_flows}), when $G$ is abelian, $p_G(M)$ yields the tension polynomial of the underlying graph of $M$ evaluated at $|G|$, counting nowhere-zero $\mathbb Z_{|G|}$-tensions.\footnote{Lluis: what is the tension polynomial? I don't think this holds for abelian as the number of faces is less than the cycle space.}
\end{example}

\begin{example}[Dihedral group] \label{ex:dihedral} Let $G=D_{2n}=\langle r,s|s^2=r^n=1, srs=r^{-1}\rangle$ be the dihedral group of order $2n$. The number of nowhere-identity $D_{2n}$-flows of a map $M$ are given by a quasi-polynomial in $n$ of period $2$: 

\begin{itemize}
\item[(1)] $n$ odd: There are two 1-dimensional irreducible representations of $D_{2n}$, the remaining $\frac{n-1}{2}$ being $2$-dimensional. This corresponds to $\frac{n+3}{2}$ conjugacy classes in $D_{2n}$. Setting $n_i=1$ for $i=1,2$ and $n_i=2$ for $i=3,4\dots, \frac{n+3}{2}$, we obtain
$$q_{D_{2n}}(M)=\sum_{A\subseteq E}(-1)^{|E\backslash A|}(2n)^{|E\backslash A|-|V|}\prod_{\mbox{\rm \tiny conn. cpts $M_j$ of $M\backslash A$}}\left(2+\frac{n-1}{2}2^{\chi(M_j)}\right).$$
\item[(2)] $n$ even: There are four 1-dimensional irreducible representations of $D_{2n}$, the remaining $\frac{n}{2}-1$ being $2$-dimensional. This corresponds to $\frac{n}{2}+3$ conjugacy classes in $D_{2n}$. Setting $n_i=1$ for $i=1,2,3,4$ and $n_i=2$ for $i=5,6\dots, \frac{n}{2}+3$, we obtain
$$q_{D_{2n}}(M)=\sum_{A\subseteq E}(-1)^{|E\backslash A|}(2n)^{|E\backslash A|-|V|}\prod_{\mbox{\rm \tiny conn. cpts $M_j$ of $M\backslash A$}}\left(4+\frac{n-2}{2}2^{\chi(M_j)}\right).$$
\end{itemize}
\end{example}

\begin{remark} The flow polynomial of a graph $\Gamma$ evaluated at $n\in\mathbb N$ enumerates the number of nowhere-zero $\mathbb Z_n$-flows of $\Gamma$. 
A {\em nowhere-zero $n$-flow} is a $\mathbb Z$-flow that takes values in $\{\pm 1,\pm 2,\dots, \pm (n-1)\}$. Tutte~\cite{T49} showed that the existence of a nowhere-zero $\mathbb Z_n$-flow of $\Gamma$ is equivalent to the existence of a nowhere-zero $n$-flow of $\Gamma$. Kochol~\cite{K02} showed that the number of nowhere-zero $n$-flows is in general larger than the number of nowhere-zero $\mathbb Z_n$-flows, but still given by a polynomial in $n$.
Example~\ref{ex:dihedral} shows that the number of nowhere-identity $D_{2n}$-flows is a quasi-polynomial in $n$ of period $2$. The dihedral group $D_{2n}$ is isomorphic to its representation in $GL_2(\mathbb Z_n)$ as
$\{\left(\begin{array}{cc} \pm 1 & x\\ 0 & 1\end{array}\right):x\in\mathbb Z_n\}.$
Defining
$$\mathbb D=\left\{\left(\begin{array}{cc} \pm 1 & x\\ 0 & 1\end{array}\right):x\in\mathbb Z\right\},$$
%and restrict values of $x$ to $\{\pm 1,\pm 2,\dots, \pm (n-1)\}$ we obtain the analogue of nowhere-zero integer $n$-flows for the dihedral group $D_{2n}$. 
it would be interesting to establish whether a nowhere-identity $D_{2n}$-flow exists if and only if a $\mathbb D$-flow exists with $x\in\{\pm 1,\pm 2,\dots, \pm (n-1)\}$ (to match Tutte's equivalence for the existence of nowhere-zero $\mathbb Z_n$-flows), and furthermore whether the latter are counted by a quasi-polynomial in $n$ of period 2 (analogous to Kochol's integer flow polynomial). 
%Tutte's characterization of the existence of nowhere-zero integer $n$-flows translates to the dihedral setting
\end{remark}

DeVos~\cite{D00} appears to be the first to have considered the problem of counting nonabelian flows. 
In \cite[Lemma 6.1.6]{D00}, DeVos argues directly that the number of nowhere-identity $D_8$-flows is equal to the number of nowhere-identity $Q_8$-flows, where $Q_8$ is the quaternion group, with presentation $\langle r,s|r^4=1, s^2=r^2=srs^{-1}r^{-1}\rangle$. 
This can be seen from Theorem~\ref{thm:n1flows} by observing that $D_{8}$ and $Q_8$ each have irreducible representations dimensions $1,1,1,1,2$.

Let $G'=\langle xyx^{-1}y^{-1}:x,y\in G\rangle$ denote the commutator subgroup of $G$.
Let us call an edge $e$ of a map connected $M$ a \emph{plane-sided bridge} if deleting $e$ from $M$ results in a disconnected map one whose components is plane, i.e., has genus zero.
The main result DeVos proves concerning nonabelian flows is the following: %characterization of maps that have a nowhere-identity $G$-flow for nonabelian group $G$: 
\begin{theorem}{\rm \cite[Theorem 6.0.7]{D00}}\label{thm:DeVos}
Let $M$ be a connected map and $G$ a nonabelian group. \begin{itemize}
\item[(1)] If $|G'|>2$ then $M$ has a nowhere-identity local $G$-flow if and only if $M$ has no plane-sided bridge. 
\item[(2)] If $|G'|=2$ and $G\not\in\{D_8,Q_8\}$ then $M$ has a nowhere-identity local $G$-flow if and only if $M$ has no odd-sized subset $B$ of bridges of $M$ such that each $e\in B$ is a plane-sided bridge of $M\backslash (B\backslash e)$. 
\item[(3)] If $G\in\{D_8,Q_8\}$ then $M$  has  a nowhere-identity local $G$-flow if $M$ has no bridge, but it is NP-complete to decide if $M$ with a bridge has  a nowhere-identity local $G$-flow.
\end{itemize}
\end{theorem}

%% file: conclusions_120.tex
%!TEX root = /Users/crovellu/Dropbox/surface_tutte/First_surface_tutte/Stutte_I_120.tex
\section{Concluding remarks}\label{sec:conclusions}

\subsection{The surface Tutte polynomial and the Kruskhal polynomial}\label{sec:not_Kruskal}%\footnote{Lluis: have slightly modified this section on 27/9/2016}

Consider the following two maps $M_1$ and $M_2$ in the $2$-torus (orientable surface of genus $2$). % with underlying graphs $\Gamma_1$ and $\Gamma_2$ respectively. 
The map $M_2$ is a $2$-vertex $5$-edge map constructed by adding a pendant edge to the $4$-loop graph on one vertex in the $2$-torus. $M_1$ is the unique map in the $2$-torus in which each of the two vertices is attached to a non-loop edge and to two loops. Both $M_1$ and $M_2$ have a non-loop edge ($e_1$ and $e_2$ respectively) that is a bridge in their underlying graph; we have %, comprising four loops and a bridge. Moreover, 
$M_1/e_1=M_2/e_2$. 

By direct computation we find:
\begin{align}
\mathcal{T}&(M_1;\mathbf{x},\mathbf{y})=
{x}^{5}x_{{2}}{y_{{0}}}^{2}+{x}^{4}x_{{2}}y_{{0}}+4\,{x}^{4}x_{{1}}y{y
_{{0}}}^{2} %\nonumber \\
+4\,{x}^{3}x_{{1}}yy_{{0}}+2\,{x}^{3}x_{{1}}{y}^{2}y_{{0}}y
_{{1}}+2\,{x}^{2}x_{{1}}{y}^{2}y_{{1}}\nonumber \\
&+4\,{x}^{3}x_{{0}}{y}^{2}{y_{{0}
}}^{2}+4\,{x}^{2}x_{{0}}{y}^{2}y_{{0}}+4\,{x}^{2}x_{{0}}{y}^{3}y_{{1}}
y_{{0}}%\nonumber \\
+4\,xx_{{0}}{y}^{3}y_{{1}}+xx_{{0}}{y}^{4}{y_{{1}}}^{2}+x_{{0}}
{y}^{4}y_{{2}}
\end{align}
and 
\begin{align}
\mathcal{T}&(M_2;\mathbf{x},\mathbf{y})=
{x}^{5}x_{{2}}{y_{{0}}}^{2}+{x}^{4}x_{{2}}y_{{0}}+4\,{x}^{4}x_{{1}}y{y
_{{0}}}^{2}+4\,{x}^{3}x_{{1}}yy_{{0}}+2\,{x}^{3}x_{{1}}{y}^{2}y_{{0}}y
_{{1}}+2\,{x}^{2}x_{{1}}{y}^{2}y_{{1}}\nonumber\\ &+4\,{x}^{3}x_{{0}}{y}^{2}{y_{{0}
}}^{2}+4\,{x}^{2}x_{{0}}{y}^{2}y_{{0}}+4\,{x}^{2}x_{{0}}{y}^{3}y_{{1}}
y_{{0}}+4\,xx_{{0}}{y}^{3}y_{{1}}+xx_{{0}}{y}^{4}y_{{0}}y_{{2}}+x_{{0}
}{y}^{4}y_{{2}},\end{align}
which are different as:
\[\mathcal{T}(M_1;\mathbf{x},\mathbf{y})-\mathcal{T}(M_2;\mathbf{x},\mathbf{y})=
xx_{{0}}{y}^{4}{y_{{1}}}^{2}-xx_{{0}}{y}^{4}y_{{0}}y_{{2}}.\]

Using~\cite[Lemma~2.2 (2)]{K11} for edges $e_1$ and $e_2$ in $M_1$ and $M_2$, respectively, together with the observation that $M_1/e_1=M_2/e_2$, we conclude that $\mathcal{K}(M_1;x,y,a,b)=\mathcal{K}(M_2;x,y,a,b)$. Therefore, the surface Tutte polynomial $\mathcal T(M;\mathbf x,\mathbf y)$ distinguishes these two maps, whereas the Kruskhal polynomial does not (and hence neither does the Las Vergnas polynomial nor the Bollob\'as--Riordan polynomial). Thus, by Proposition~\ref{prop:specs}, $\mathcal T(M;\mathbf x,\mathbf y)$ strictly refines the partition on maps induced by the Kruskhal polynomial. %\footnote{ñLluis: added:  Thus, by Proposition~\ref{prop:specs}, $\mathcal T(M;\mathbf x,\mathbf y)$ strictly refines the partition on maps induced by the Kruskhal polynomial. }

Since $\mathcal{T}(M;\mathbf{x},\mathbf{y})$ has $4+2g(M)$ variables, this latter fact that the distinguishing power of $\mathcal{T}(M;\mathbf{x},\mathbf{y})$ (partition of maps into equivalence classes according to the value of their surface Tutte polynomial) refines that of $\mathcal K(M;x,y,a,b)$ is perhaps unsurprising when $g(M)>0$. % it is perhaps unsurprising that the distinguishing power of $\mathcal{T}(M;\mathbf{x},\mathbf{y})$ (partition of maps into equivalence classes according to the value of their surface Tutte polynomial) refines that of $\mathcal K(M;x,y,a,b)$. 
%However, these additional variables do not separate those maps in the case of connected plane maps.\footnote{Indeed, using 
On the other hand, by Proposition~\ref{prop:surface_Tutte_plane}, given two plane maps, the two underlying graphs have the same Tutte polynomial if and only if the two plane maps have the same surface Tutte polynomial (a polynomial in 4 variables). %, whence the additional variables do not increase the distinguishing power of the polynomial.

%If we use the evaluation $x_g=a^g$, $y_g=b^g$ to obtain a 4-variable polynomial as in $\mathcal{K}(M;X,Y,a,b)$, we observe that both polynomials for the surface Tutte become the same. 

In light of the specializations of Section~\ref{sec:specializations}, it seems natural to consider the $4$-variable specialization
 $\mathcal Q(M;x,y,a,b)$ of the surface Tutte polynomial (setting $x_g=a^g$, $y_g=b^g$ for $g=0,1,\dots$) and ask for its distinguishing power with respect to the Kruskhal polynomial. As yet we have no counterxample to following:
\begin{problem}\label{prob:QK} Let $M$ and $M'$ be maps. %, each of which is a 2-cell embedding of a graph $\Gamma$, respectively in surfaces $\Sigma$ and $\Sigma'$. 
Is it true that
 $$\mathcal Q(M;x,y,a,b)=\mathcal Q(M';x,y,a,b)\quad\mbox{if and only if}\quad \mathcal K(M;x,y,a,b)=\mathcal K(M';x,y,a,b)\:?$$
\end{problem}

\subsection{Non-orientable surfaces}
It is natural to ask whether $\mathcal{T}(M;\mathbf{x},\mathbf{y})$ can be extended from maps to graphs embedded in non-orientable surfaces in such a way that there are evaluations of this extended polynomial which enumerate nowhere-identity local $G$-flows and nowhere-identity local $G$-tensions (the definition of which extend to non-orientable embeddings). 
The paper~\cite{MY05} contains results that suggest a way forward %it is likely that something similar can be done in this case 
and we expect to report on this in a future paper.

\subsection{Deletion/contraction recursion}

The Tutte polynomial is universal for graph invariants satisfying a deletion-contraction recursion for ordinary edges, loops and bridges~\cite{OW79}: together with the boundary condition that it takes the value $1$ on edgeless graphs this recurrence determines the polynomial.
 
The Bollob\' as--Riordan polynomial and Krushkal both satisfy a similar recurrence, except reduction is only for ordinary edges and bridges, with boundary conditions determined by their value on bouquets~\cite{BR01, K11}. 

What about other specializations of the surface Tutte polynomial, such as $\mathcal{Q}(M;a,y,a,b)$?
%In Section~\ref{sec:independence} we have seen that no evaluation of $\mathcal{T}(M;\mathbf{x},\mathbf{y})$ gives the Tutte polynomial %(up to a prefactor depending on $v(M), e(M), k(M)$ and $g(M)$) 
%when the values given to the variables $\mathbf x$ and $\mathbf y$ are not allowed to depend on the map $M$. Hence $\mathcal{T}(M;\mathbf{x},\mathbf{y})$ will not satisfy a deletion-contraction recurrence as simple as that satisfied by the Tutte polynomial.

As we have already remarked, contraction of loops in maps does not behave in the same way as for graphs.
For example, taking two loops $e_1,e_2$ attached to a single vertex $v$ embedded in the torus (so the vertex rotation at $v$ is given by $(e_1,e_2,e_1,e_2)$), contracting a loop results in a plane map consisting of two vertices connected by an edge. 
If this same graph is embedded in the plane (so the vertex rotation at $v$ is given by $(e_1,e_1,e_2,e_2)$), then contracting a loop results in a the disjoint union of a vertex and an edge, a disconnected plane map.
%We note that it is important to be able to contract loops, as by duality this corresponds to deleting edges that belong to exactly one face.
%It is quickly realized then that any deletion-contraction recurrence for $\mathcal{Q}(M;a,y,a,b)$ will not be able to remove  

In subsequent work we aim to determine what sort of deletion-contraction recurrence is satisfied by $\mathcal{T}(M;\mathbf{x},\mathbf{y})$, $\mathcal{Q}(M;a,y,a,b)$ and other polynomials derived from the surface Tutte polynomial, and to describe the boundary conditions for them. For example, for the number of nowhere-identity local $G$-flows of $M$, $q_G(M)$, it is straightforward to see that $q_G(M)$ satisfies a recurrence for non-loops and, after establishing a four-term relation for $q_G(M)$ on bouquets, that its value is determined by its value on bouquets whose chord diagram is of the form $D_{i,j}$ in the notation of~\cite[Lemma 5]{BR01}.

%study the effect of deletion and contraction of edges of a map $M$ with respect to such parameters as $n(M), n^*(M)$ and $g(M)$,  %such as the nullity, rank and genus of a map.
%in order to 
\section*{Acknowledgements}
We thank Matt DeVos for sending us a copy of his thesis.